\documentclass[11pt,a4paper]{scrartcl}

\usepackage[T1]{fontenc}
\usepackage[utf8]{inputenc}
\usepackage{lmodern}
\usepackage{amsmath,amssymb,amsthm,mathtools}
\usepackage{microtype}
\usepackage{enumitem}
\usepackage[colorlinks=true,linkcolor=blue,citecolor=blue,urlcolor=blue]{hyperref}
\usepackage{tikz}
\usetikzlibrary{arrows.meta,decorations.pathreplacing,positioning}

\newtheorem{theorem}{Theorem}[section]
\newtheorem{lemma}[theorem]{Lemma}
\newtheorem{corollary}[theorem]{Corollary}
\theoremstyle{remark}
\newtheorem{remark}[theorem]{Remark}

\newcommand{\R}{\mathbb R}
\newcommand{\Rp}{\mathbb R_+}
\newcommand{\Borel}{\mathcal B}
\newcommand{\Law}{\operatorname{Law}}
\newcommand{\supp}{\operatorname{supp}}
\newcommand{\diam}{\operatorname{diam}}

\title{Realising atomless laws as distance distributions on metric measure spaces}
\author{Christoph Th\"ale and Philipp Tuchel}
\date{}

\begin{document}
	\maketitle
	
	\begin{abstract}
		{\noindent\small We study which probability measures on $[0,\infty)$ can occur as the distribution of the distance between two independent points sampled from a complete separable metric space equipped with a Borel probability measure.  We prove that an atomless Borel probability measure can be realised in this way if and only if its support contains the origin. This settles the conjecture of Aldous, Blanc, and Curien for absolutely continuous laws and also covers atomless measures that are singular with respect to Lebesgue measure. Moreover, for every $L>1$, we show that the realising metric may be chosen $L$-bi-Lipschitz equivalent to an ultrametric. The proof constructs the space from compact components whose mutual distances produce prescribed parts of the measure, while distances within the components are assigned to smaller scales.}
	\end{abstract}
	
	
	\section{Introduction and main results}
	
	Let $(S,d)$ be a complete separable metric space equipped with a Borel
	probability measure $\mu$.  If $X$ and $Y$ are independent random
	elements with common law $\mu$, then $d(X,Y)$ is a nonnegative random
	variable whose distribution is a natural invariant of the metric
	probability space $(S,d,\mu)$.  The \emph{distance problem} asks for the
	inverse characterisation: which probability measures on
	$\Rp=[0,\infty)$ arise in this way?
	
	The problem originates in a question posed by Aldous on MathOverflow
	concerning distance distributions in Hilbert spaces \cite{AldousMO2022}.
	Aldous, Blanc, and Curien \cite{AldousBlancCurien2025} subsequently considered it in the setting of
	measured metric spaces.  They showed that the
	condition that the support contain the origin does not suffice for
	arbitrary laws with atoms \cite[Proposition~2.2]{AldousBlancCurien2025}
	and asked whether it is sufficient for absolutely continuous laws
	\cite[Problem~1.7]{AldousBlancCurien2025}.
	The same question was stated by Gross in \cite[Problem~1.3]{Gross2025}.
	The support condition is in fact necessary. 
	Namely, if $d(X,Y)\ge a>0$ almost surely,
	then every open ball of radius $a/2$ has $\mu$-measure zero, since two
	points in such a ball have distance less than $a$.  A separable metric
	space is covered by countably many balls of that radius, which would give
	$\mu(S)=0$.  Thus the support of every distance distribution arising from a metric
	probability space must contain the origin.
	
	Previous results cover several special cases of the distance problem.
	Aldous, Blanc, and Curien proved that every finitely supported probability
	measure whose support contains the origin can be realised as a distance
	distribution \cite[Theorem~1.4]{AldousBlancCurien2025}.  They also showed
	that, for every continuous probability density $f$ which is positive near
	the origin and every $\varepsilon>0$, there is a continuous probability
	density $g$ satisfying $g\le(1+\varepsilon)f$ such that the measure
	$g(t)\,dt$ can be realised as a distance distribution
	\cite[Proposition~1.5]{AldousBlancCurien2025}.  Gross extended the
	discrete result to countably supported laws
	\cite[Theorem~1.1]{Gross2025} and proved that every probability measure
	on $[0,1]$ whose density is bounded above and below by positive constants
	can be realised as a distance distribution
	\cite[Theorem~1.2]{Gross2025}.
	
	To state our main result, we call a Borel probability measure
	\(\Theta\) on \(\Rp\) \emph{feasible} if there exist a complete
	separable metric space \((S,d)\), a Borel probability measure \(\mu\)
	on \(S\), and independent random elements \(X\) and \(Y\) with common
	law \(\mu\) such that \(d(X,Y)\) has distribution \(\Theta\).  We call
	\((S,d,\mu)\) a \emph{realisation} of \(\Theta\).  It is called a
	compact realisation if \((S,d)\) is compact.  The following
	theorem gives a complete answer for atomless laws.
	
	\begin{theorem}
		\label{thm:main}
		Let $\Theta$ be an atomless Borel probability measure on $\Rp$.  Then
		$\Theta$ is feasible if and only if $0\in\supp\Theta$.
	\end{theorem}
	
	Theorem~\ref{thm:main} also applies to atomless measures that are
	singular with respect to Lebesgue measure.  For absolutely continuous
	measures, it proves the conjecture in
	\cite[Problem~1.7]{AldousBlancCurien2025} without any regularity or
	boundedness assumption on the density.
	
	\begin{corollary}
		\label{cor:abc}
		Let $\Theta=f(t)\,dt$ be an absolutely continuous Borel probability
		measure on $\Rp$ with $0\in\supp\Theta$.  Then $\Theta$ is feasible.
	\end{corollary}

	The theorem concerns only the distribution of the distance between two
	independently sampled points.  This is the simplest marginal of the
	random distance matrix associated with a metric probability space.  A different inverse question asks to what extent this marginal determines
	the underlying metric probability space within a prescribed class, see
	\cite{MemoliNeedham2022}. At
	the other end, the distribution of the infinite random distance matrix
	$(d(X_i,X_j))_{i,j\geq 1}$,
	where the random elements \(X_i\), \(i\geq 1\), are independent with law
	\(\mu\), determines the metric measure space, on the support of its
	measure, up to measure-preserving isometry by Gromov's reconstruction
	theorem \cite{Gromov1999}, see also
	\cite{Kondo2005,Vershik1998,Vershik2004}.  The laws of all finite sampled
	distance matrices determine the Gromov--weak topology
	\cite{GrevenPfaffelhuberWinter2009}.
	Between the two-point distribution considered here and the full random
	distance matrix lies, in particular, the joint distribution $\bigl(d(X_1,X_2),d(X_1,X_3),d(X_2,X_3)\bigr)$.
	The paper \cite{AldousBlancCurien2025} identifies the characterisation of such
	three-point laws as a natural next problem.  At this level, the triangle inequality
	already imposes nontrivial restrictions.

	There is a simple model which helps to explain the proof and also appears
	in \cite[Section~1.2, Example~(6)]{Gross2025}.  For $t\geq 0$ let
	$F(t)=\Theta([0,t])$ and let $\lambda$ be the probability measure on
	$\Rp$ defined by $\lambda([0,t])=\sqrt{F(t)}$.  Since $\Theta$ is
	atomless, so is $\lambda$.  Define
	\[
	d_0(x,y)=
	\begin{cases}
		0, & x=y,\\
		\max\{x,y\}, & x\ne y.
	\end{cases}
	\]
	This is an ultrametric on $\Rp$.  If $X$ and $Y$ are independent with
	law $\lambda$, then
	\[
	\mathbb P(d_0(X,Y)\le t)=\lambda([0,t])^2=F(t),
	\]
	because the diagonal has $(\lambda\otimes\lambda)$-measure zero.  Thus
	$d_0(X,Y)$ has the prescribed law $\Theta$.
	
	This construction does not give a metric probability space in the sense
	used here.  Every $x>0$ is isolated in the $d_0$-topology, so
	$(\Rp,d_0)$ is not separable.  Moreover, every subset of $(0,\infty)$ is
	$d_0$-open.  Hence the Borel $\sigma$-field generated by $d_0$ strictly
	contains the usual Borel $\sigma$-field on $\Rp$, on which $\lambda$ was
	defined. A separable ultrametric space cannot produce an atomless distance law,
	since every such space has at most countably many positive distance
	values according to \cite[Theorem~4.12]{DovgosheyShcherbak2022}.
	
	The proof replaces the isolated points of the
	above model by compact metric spaces.  Distances between different
	components reproduce prescribed parts of $\Theta$, while the internal
	distances created by each component are placed at smaller scales and
	subtracted from the mass available there.  This is related to the
	tree-based construction in \cite[Theorem 1.4 and Sections 3-4]{AldousBlancCurien2025}.  There, finite
	discrete laws are realised by weighted trees, and scaled compact spaces
	are grafted onto their leaves to replace discrete distance values by
	measures on short intervals.  The present construction uses the same
	distinction between distances within and between components, but
	controls the internal distance laws through the mass of the target
	measure nearer the origin.  Iterating this procedure yields a
	realisation when $\Theta$ has bounded support.
	
	The construction for countably supported laws in \cite[Section 2.1]{Gross2025} is
	also hierarchical.  The underlying space consists of sequences, and
	the distance between two sequences is determined by the last coordinate
	at which they differ.  In the local construction below, the first level
	at which two points belong to different partition cells similarly
	determines the scale of their distance, while transport maps determine
	its value within that scale.  The result for densities in
	\cite[Section 3]{Gross2025} follows instead by composing a given metric with a
	subadditive metric-preserving function.  Here the metric is constructed
	directly, and the target measure need not have a density.  For an
	unbounded law, the part outside a bounded interval is divided into an
	ordered sequence of measures supported on short intervals and then
	combined with a compact realisation of the remaining part.
	
	The construction also gives information about the geometry of the
	realising space $(S,d)$.  If the target measure has bounded support, the
	measure may be chosen to have full support on a space homeomorphic to
	the Cantor space $\{0,1\}^\mathbb{N}$.  For an unbounded target measure, the realising space
	may be chosen proper and homeomorphic to a countable disjoint union of
	Cantor spaces.  More quantitatively, for every $L>1$ the metric may be
	chosen together with an ultrametric $u$ such that
	\[
	u(x,y)\le d(x,y)\le L u(x,y),
	\qquad x,y\in S.
	\]
	These consequences and some related questions are discussed in
	Section~\ref{sec:geometry}.
	
	\medspace
	
	The paper is organised as follows.  Section~\ref{sec:Notation-Aux-Results}
	collects the notation and auxiliary results used later.
	Section~\ref{sec:local-construction} contains the local construction.
	The compactly supported case is proved in
	Section~\ref{sec:compact-case}, and the general case in
	Section~\ref{sec:general-case}.  Further properties of the realising
	spaces and some related questions are considered in
	Section~\ref{sec:geometry}.
	
	\section{Notation and auxiliary results}\label{sec:Notation-Aux-Results}
	
	\subsection{Notation}
	
	We write $\Rp=[0,\infty)$ and take $\mathbb N=\{1,2,\ldots\}$.
	For a family of sets $(X_i)_{i\in I}$, we write
	$\bigsqcup_{i\in I}X_i$ for their disjoint union, whose elements are the
	pairs $(x,i)$ with $i\in I$ and $x\in X_i$. Let
	$(X,d)$ be a metric space.  For $x\in X$ and $r>0$, we write $B_d(x,r)=\{y\in X:d(x,y)<r\}$.
	When the metric is clear from the context, we just write $B(x,r)$. 
	For $\varepsilon>0$, a finite set $F\subseteq X$ is called an
	$\varepsilon$-net of $(X,d)$ if $X\subseteq\bigcup_{x\in F}B_d(x,\varepsilon)$.
	The metric space $(X,d)$ is called totally bounded if it has a finite
	$\varepsilon$-net for every $\varepsilon>0$.
	For a
	nonempty set $A\subseteq X$, its diameter is
	$\diam_d A=\sup\{d(x,y):x,y\in A\}$.  Again, we write $\diam A$ when the metric
	is understood.  For subsets of $\R$, the diameter is taken with respect
	to the Euclidean metric unless another metric is specified.  If $T$ is
	a topological space, its Borel $\sigma$-field is denoted by $\Borel(T)$. For a subset $A$ of a topological space $T$, we denote its closure in $T$ by $\overline A$.
	
	Let $(E,\mathcal E)$ be a measurable space, let $\nu$ be a finite
	measure on $(E,\mathcal E)$, and let $A\in\mathcal E$.  The restriction
	of $\nu$ to $A$ is the finite measure $\nu|_A$ on $(E,\mathcal E)$
	defined by $\nu|_A(B)=\nu(A\cap B)$, $B\in\mathcal E$.
	If $\nu(A)>0$, the corresponding normalised restriction is denoted by
	$\nu|_A/\nu(A)$.  If $\alpha$ and $\beta$ are finite measures on the
	same measurable space $(E,\mathcal E)$, then $\alpha\le\beta$ means that
	$\alpha(B)\le\beta(B)$ for every $B\in\mathcal E$.  In this case,
	$\beta-\alpha$ denotes the finite positive measure given by
	$(\beta-\alpha)(B)=\beta(B)-\alpha(B)$ for $B\in\mathcal E$. 
	
	A set $A\in\mathcal E$ is an atom of a finite measure $\nu$ on
	$(E,\mathcal E)$ if $\nu(A)>0$ and, for every $B\in\mathcal E$ with
	$B\subseteq A$, either $\nu(B)=0$ or $\nu(A\setminus B)=0$.  The measure
	$\nu$ is called atomless if it has no atoms.  For a finite Borel measure
	on $\Rp$, this is equivalent to
	$\nu(\{t\})=0$ for every $t\in\Rp$.
	
	Let $T$ be a topological space and let $\nu$ be a Borel measure on $T$.
	Its support is
	\[
	\supp\nu
	=
	\{x\in T:\nu(U)>0\text{ for every open neighbourhood }U\text{ of }x\}.
	\]
	This is a closed subset of $T$.  For a Borel measure $\nu$ on $\Rp$, the
	condition $0\in\supp\nu$ is equivalent to
	$\nu([0,\varepsilon])>0$ for every $\varepsilon>0$.
	
	Let $(E,\mathcal E)$ and $(F,\mathcal F)$ be measurable spaces, let
	$\mu$ be a measure on $(E,\mathcal E)$, and let
	$\phi:E\to F$ be measurable.  The pushforward $\phi_\#\mu$ is the
	measure on $(F,\mathcal F)$ defined by $(\phi_\#\mu)(B)=\mu(\phi^{-1}(B))$, $B\in\mathcal F$.
	If $Z$ is an $E$-valued random variable on a probability space
	$(\Omega,\mathcal A,\mathbb P)$, then
	$\Law(Z)=Z_\#\mathbb P$.  For a probability measure $\nu$ on
	$(E,\mathcal E)$, the notation $Z\sim\nu$ means that
	$\Law(Z)=\nu$. For a map $\phi$, we denote its domain by $\operatorname{dom}\phi$.
	
	Let $X$ be a Borel subset of $\R$, equipped with the Borel structure
	inherited from $\R$.  A metric $d$ on $X$ is called Borel if the map $d:(X\times X,\Borel(X)\otimes\Borel(X))
	\to(\Rp,\Borel(\Rp))$
	is measurable.  Throughout, a metric probability space is a triple
	$(S,d,\mu)$ such that $(S,d)$ is a complete separable metric space and
	$\mu$ is a probability measure on the Borel $\sigma$-field generated
	by the metric topology of $(S,d)$.  Sums over empty index sets are
	understood to be zero.
	
	We use standard measure-theoretic terminology as in
	\cite{Fremlin2001}. In particular, whenever a finite Borel measure on $\Rp$ is used as a measure on $\R$,
	it is understood to be extended by zero to $(-\infty,0)$.
	
	\subsection{Atomless partitions}
	
	Let $\nu$ be a finite Borel measure on $\R$.  A Borel set
	$\Omega\subseteq\R$ has full $\nu$-measure if
	$\nu(\R\setminus\Omega)=0$.  A finite Borel partition of $\Omega$ is a
	finite family of pairwise disjoint Borel sets whose union is $\Omega$.
	The members of such a partition are called its cells.  A Borel
	partition $\mathcal P$ refines a Borel partition $\mathcal R$ if every
	cell of $\mathcal P$ is contained in a cell of $\mathcal R$.
	
	We shall repeatedly use the following consequences of the divisibility
	of atomless finite measures, see, for example,
	\cite[Proposition~215D]{Fremlin2001}.  We include the proof because the
	partitions must consist of Borel sets and because the quantitative
	estimate in part~\textnormal{(ii)} will be needed below.
	
	\begin{lemma}
		\label{lem:division}
		Let $\nu$ be an atomless finite Borel measure on $\R$.
		\begin{enumerate}[label=\textnormal{(\roman*)}]
			\item If $E\in\Borel(\R)$ and $0\le a\le\nu(E)$, there is a Borel set
			$E_a\subseteq E$ with $\nu(E_a)=a$.
			\item If $\mathcal R$ is a finite Borel partition of a full-$\nu$-measure
			Borel set into cells of positive measure and $\delta>0$, then there is a
			finite Borel partition $\mathcal P$ of a full-$\nu$-measure Borel set
			such that every cell of $\mathcal P$ has positive measure, every cell of
			$\mathcal P$ is contained in a cell of $\mathcal R$, and
			$\max_{C\in\mathcal P}\nu(C)\le\delta$.  If $\nu$ is a probability
			measure, then
			$\sum_{C\in\mathcal P}\nu(C)^2\le\delta$.
			\item Every finite Borel measure $\lambda$ satisfying $0\le\lambda\le\nu$
			is atomless.
		\end{enumerate}
	\end{lemma}
	\begin{proof}
		For a Borel set $E$, the function
		$G_E(t)=\nu(E\cap(-\infty,t])$ is continuous because $\nu$ has no
		atoms.  Its limits at $-\infty$ and
		$+\infty$ are $0$ and $\nu(E)$, respectively.  For $a=0$ take
		$E_a=\varnothing$, and for $a=\nu(E)$ take $E_a=E$.  If
		$0<a<\nu(E)$, the intermediate value theorem gives $t\in\R$ such that
		$G_E(t)=a$. Then $E_a=E\cap(-\infty,t]$ has the required mass.  This
		proves (i).
		
		For (ii), divide every cell $C\in\mathcal R$ into
		$k_C=\lceil\nu(C)/\delta\rceil$ Borel pieces of equal mass, using (i)
		successively.  
		The resulting family
		is a finite Borel refinement and every cell has mass at most $\delta$.
		If $\nu$ is a probability measure,
		\[
		\sum_{C\in\mathcal P}\nu(C)^2
		\le
		\Bigl(\max_{C\in\mathcal P}\nu(C)\Bigr)
		\sum_{C\in\mathcal P}\nu(C)
		\le\delta.
		\]
Finally, suppose that $0\le\lambda\le\nu$.  Then
$\lambda(\{x\})=0$ for every $x\in\R$.  For every Borel set
$E\subseteq\R$, the function
$t\mapsto\lambda(E\cap(-\infty,t])$ is therefore continuous.  If
$\lambda(E)>0$, the intermediate value theorem divides $E$ into two
Borel sets of strictly positive $\lambda$-measure.  Hence $E$ is not
an atom, and $\lambda$ is atomless.
	\end{proof}
	
	\subsection{Gluing metric spaces}
	
	The proof repeatedly combines a finite ordered family of metric spaces.
	Their internal metrics are retained, while the distance between a point
	in one component and a point in a later component is prescribed by a
	$1$-Lipschitz function on the earlier component.  The following lemma
	gives conditions under which this construction defines a metric.
	
	\begin{lemma}
		\label{lem:annular-amalgamation}
		Let $R>0$ and $N\in\mathbb N$, and let
		$(X_i,d_i)$, $i=1,\ldots,N$, be metric spaces
		with \(\diam X_i\le R\).  We order the components by their indices and call $X_i$ earlier than
		$X_j$ if $i<j$. For \(i=1,\ldots,N-1\), let
		\(\phi_i:X_i\to[R,2R]\) be \(1\)-Lipschitz.  On the disjoint union
		\(X=\bigsqcup_{i=1}^N X_i\), retain \(d_i\) on each \(X_i\) and, for
		\(1\le i<j\le N\), set
		\[
		d(x,y)=d(y,x)=\phi_i(x),
		\qquad x\in X_i,\quad y\in X_j.
		\]
		Then \(d\) is a metric on \(X\).
	\end{lemma}
	\begin{proof}
		Only the triangle inequality needs to be proved.  Suppose first that
		$x,x'\in X_i$.  If $y\in X_j$ for some $j>i$, then
		\[
		|d(x,y)-d(x',y)|
		=
		|\phi_i(x)-\phi_i(x')|
		\le d_i(x,x'),
		\]
		while
		\[
		d_i(x,x')
		\le R
		\le d(x,y)+d(x',y).
		\]
		These inequalities give the three triangle inequalities for
		$x,x',y$.  If $y\in X_j$ for some $j<i$, then
		$d(x,y)=d(x',y)=\phi_j(y)\ge R$.  Hence
		\[
		|d(x,y)-d(x',y)|=0\le d_i(x,x')
		\]
		and
		\[
		d_i(x,x')\le R\le d(x,y)+d(x',y).
		\]
		
		If the three points all lie in distinct components, let $x$ belong to
		the earliest component and denote the other two points by $y$ and $z$.
		Then $d(x,y)=d(x,z)\ge R$, while $d(y,z)\le2R$.  Hence
		\[
		d(y,z)\le2R\le d(y,x)+d(x,z).
		\]
		The other two triangle inequalities follow from
		$d(x,y)=d(x,z)$.
	\end{proof}
	
The preceding lemma constructs a metric on an ordered union of
components.  When the components carry probability measures and are
assigned weights, the resulting distance law can be computed by
separating pairs of points from the same component and from different
components.  The next lemma records this computation for the two
orders used below.  In part~\textnormal{(i)}, the distance is determined
by the point in the earlier component.  In part~\textnormal{(ii)}, it
is determined by the point in the later component. In what follows, measures on individual components are identified with their
extensions by zero to the full disjoint union.
	
	\begin{lemma}
		\label{lem:distance-identity}
		Let \(I=\{1,\ldots,N\}\) for some \(N\in\mathbb N\), or let
		\(I=\mathbb N\).  Let \((X_i,d_i,\mu_i)\), \(i\in I\), be an ordered
		family of pairwise disjoint metric probability spaces, and write
		\(\nu_i=\Law(d_i(U_i,V_i))\), where \(U_i,V_i\) are independent with
		law \(\mu_i\).  In each of the following cases, let \(d\) be a metric on the indicated
		disjoint union whose restriction to each component is its given metric,
		and assume that all components are Borel subsets of the resulting metric
		space.
		\begin{enumerate}[label=\textnormal{(\roman*)}]
			\item The ordered family may be augmented by one additional component
			$(X_\ast,d_\ast,\mu_\ast)$.  This single component is placed after all
			components $X_i$, $i\in I$. Let $p_i>0$, $i\in I$.  If $X_\ast$ is present, let
			$p_\ast>0$, assume that $p_\ast+\sum_{i\in I}p_i=1$,
			and put
			\[
			\mu=\sum_{i\in I}p_i\mu_i+p_\ast\mu_\ast,
			\qquad
			\nu_\ast=\Law(d_\ast(U_\ast,V_\ast)),
			\]
			where $U_\ast$ and $V_\ast$ are independent with law $\mu_\ast$.
			If $X_\ast$ is absent, assume that
			$\sum_{i\in I}p_i=1$, put
			$\mu=\sum_{i\in I}p_i\mu_i$, set $p_\ast=0$, and let
			$\nu_\ast$ be the zero measure on $\Rp$.
			
			Let $\phi_i:X_i\to\Rp$, $i\in I$, be measurable.  Suppose that
			$d(x,y)=\phi_i(x)$ whenever $x\in X_i$ and $y$ belongs to a later
			component, and put $\eta_i=(\phi_i)_\#\mu_i$.  If $U$ and $V$ are
			independent with law $\mu$, then
			\[
			\Law(d(U,V))
			=
			p_\ast^2\nu_\ast
			+\sum_{i\in I}p_i^2\nu_i
			+2\sum_{i\in I}p_i\Big(
			p_\ast+\sum_{\substack{j\in I\\j>i}}p_j
			\Big)\eta_i.
			\]
			
			\item Let $(X_0,d_0,\mu_0)$ be an additional initial component, which
			is placed before $X_i$ for every $i\in I$.  Let
			$p_i>0$, $i\in I\cup\{0\}$, and assume that $p_0+\sum_{i\in I}p_i=1$.
			Put
			\[
			\mu=p_0\mu_0+\sum_{i\in I}p_i\mu_i,
			\qquad
			\nu_0=\Law(d_0(U_0,V_0)),
			\]
			where $U_0$ and $V_0$ are independent with law $\mu_0$.
			
			Let $\phi_i:X_i\to\Rp$, $i\in I$, be measurable and suppose that
			$d(x,y)=\phi_i(x)$ whenever $x\in X_i$ and $y$ belongs to an earlier
			component.  If $U$ and $V$ are independent with law $\mu$, then
			\[
			\Law(d(U,V))
			=
			p_0^2\nu_0
			+\sum_{i\in I}p_i^2\nu_i
			+2\sum_{i\in I}p_i\Big(
			p_0+\sum_{\substack{j\in I\\j<i}}p_j
			\Big)(\phi_i)_\#\mu_i.
			\]
		\end{enumerate}
	\end{lemma}
	
	\begin{proof}
			The distance law is the pushforward
			$d_\#(\mu\otimes\mu)$.  For every $i\in I$, the restriction of
			$\mu\otimes\mu$ to $X_i\times X_i$ has pushforward
			$p_i^2\nu_i$.  The terminal component in part~\textnormal{(i)} and
			the initial component in part~\textnormal{(ii)} contribute
			$p_\ast^2\nu_\ast$ and $p_0^2\nu_0$, respectively.
			
			Consider part~\textnormal{(i)}.  For $i\in I$, let $L_i$ be the
			union of all components placed after $X_i$ and put $A_i=(X_i\times L_i)\cup(L_i\times X_i)$.
			The sets $A_i$, $i\in I$, are pairwise disjoint and cover all pairs
			of points belonging to distinct components.  Since the distance on
			$A_i$ is determined by $\phi_i$, we have
			\[
			d_\#\bigl((\mu\otimes\mu)|_{A_i}\bigr)
			=
			2p_i\mu(L_i)(\phi_i)_\#\mu_i
			=
			2p_i\Big(
			p_\ast+\sum_{\substack{j\in I\\j>i}}p_j
			\Big)\eta_i.
			\]
			Summing these identities over $i\in I$ proves
			part~\textnormal{(i)}.
			
			For part~\textnormal{(ii)}, the same argument applies with
			\[
			L_i=X_0\sqcup
			\bigsqcup_{\substack{j\in I\\j<i}}X_j,
			\qquad i\in I.
			\]
			In this case,
			$\mu(L_i)=p_0+\sum_{\substack{j\in I\\j<i}}p_j$, which gives the
			stated formula.
	\end{proof}
	
	\subsection{Choosing the component weights}
	
	In the applications below, part of the target measure is written as
	$\sum_{i\in I}c_i\eta_i$, where each $\eta_i$ is a probability measure.
	The number $p_i$ is the mass assigned to the component realising
	$\eta_i$, and $p_0$ is the mass of the initial component $X_0$.  By
	Lemma~\ref{lem:distance-identity}(ii), the required coefficient identity
	is
	\[
	c_i=2p_i\Big(
	p_0+\sum_{\substack{j\in I\\j<i}}p_j
	\Big).
	\]
	The next lemma chooses such weights and controls
	$\sum_{i\in I}p_i^2$.  For a finite family, reversing the order gives
	the coefficients in Lemma~\ref{lem:distance-identity}(i), with $p_0$
	denoted by $p_\ast$.
	
	\begin{lemma}
		\label{lem:mass-selection}
		Let $I=\{1,\ldots,N\}$ for some $N\in\mathbb N$, or let
		$I=\mathbb N$.  Let $\alpha\in(0,1)$ and let
		$(c_i)_{i\in I}$ be a family of positive numbers satisfying $\sum_{i\in I}c_i=1-\alpha$.
		There is a constant $\gamma_0(\alpha)>0$ such that, if
		$\sup_{i\in I}c_i\le\gamma_0(\alpha)$, then there exist $p_0>0$
		and a family $(p_i)_{i\in I}$ of positive numbers satisfying
		\[
		\begin{alignedat}{2}
			&p_0+\sum_{i\in I}p_i
			=1,
			&\qquad
			&p_0^2+\sum_{i\in I}p_i^2
			=\alpha,
			\\
			&c_i
			=2p_i\Big(
			p_0+\sum_{\substack{j\in I\\j<i}}p_j
			\Big),
			\quad i\in I,
			&\qquad
			&\sum_{i\in I}p_i^2
			\le
			\frac{1-\alpha}{\alpha}\sup_{i\in I}c_i.
		\end{alignedat}
		\]
	\end{lemma}
	\begin{proof}
		Put $C=1-\alpha$.  For $p>0$, define
		\[
		T_1(p)=p,\qquad
		p_i(p)=\frac{c_i}{2T_i(p)},\qquad
		T_{i+1}(p)=T_i(p)+p_i(p),
		\qquad i\in I.
		\]
		If $I=\{1,\ldots,N\}$, put $T(p)=T_{N+1}(p)$.  In this case,
		$T$ is continuous by the recursion.  Suppose now that $I=\mathbb N$.
		Since $T_i(p)\ge p$, we have
		\[
		\sum_{i=1}^{\infty}p_i(p)
		\le
		\frac{1}{2p}\sum_{i=1}^{\infty}c_i
		=
		\frac{C}{2p}.
		\]
		Hence the limit $T(p)=\lim_{i\to\infty}T_i(p)$
		exists and is finite.  For every compact interval
		$[p_-,p_+]\subset(0,\infty)$,
		\[
		\begin{aligned}
			\sup_{p\in[p_-,p_+]}|T(p)-T_n(p)|
			&=
			\sup_{p\in[p_-,p_+]}
			\sum_{i=n}^{\infty}p_i(p)
			\le
			\frac{1}{2p_-}\sum_{i=n}^{\infty}c_i
			\longrightarrow0
		\end{aligned}
		\]
		as $n\to\infty$.  Thus $T_n$ converges to $T$ uniformly on every
		compact subinterval of $(0,\infty)$.  Since every $T_n$ is continuous,
		the function $T$ is continuous.
		
		The identity $T_{i+1}(p)^2-T_i(p)^2=c_i+p_i(p)^2$
		follows from the definition of $p_i(p)$.  Summing this identity over $i=1,\ldots,N$ and letting
		$N\to\infty$ in the countable case gives
		\begin{equation}\label{eq:TpQuadrat}
		T(p)^2=p^2+C+\sum_{i\in I}p_i(p)^2.
		\end{equation}
		In particular, $T(\sqrt{\alpha})>1$.  Set $\gamma_0(\alpha)=\frac{3\alpha^2}{8C}$.
		If $\sup_{i\in I}c_i\le\gamma_0(\alpha)$, then, at
		$p=\sqrt{\alpha}/2$,
		\[
		\sum_{i\in I}p_i(p)^2
		\le
		\frac{1}{4p^2}\sum_{i\in I}c_i^2
		\le
		\frac{C}{\alpha}\sup_{i\in I}c_i
		\le\frac{3\alpha}{8}.
		\]
		Consequently,
		$T(\sqrt{\alpha}/2)^2\le1-3\alpha/8<1$.  By continuity, there is
		$p_0\in(\sqrt{\alpha}/2,\sqrt{\alpha})$ such that $T(p_0)=1$.
		Put $p_i=p_i(p_0)$ for $i\in I$.  Since $T_i(p_0)=p_0+\sum_{\substack{j\in I\\j<i}}p_j$,
		the recursion gives
		\[
		p_0+\sum_{i\in I}p_i=1,
		\qquad
		c_i=2p_i\Big(
		p_0+\sum_{\substack{j\in I\\j<i}}p_j
		\Big),
		\quad i\in I.
		\]
		Evaluating \eqref{eq:TpQuadrat} at $p_0$ gives
		$p_0^2+\sum_{i\in I}p_i^2=\alpha$.  Finally,
		\[
		\sum_{i\in I}p_i^2
		\le
		\frac{1}{4p_0^2}\sum_{i\in I}c_i^2
		\le
		\frac{C}{\alpha}\sup_{i\in I}c_i,
		\]
		which proves the required estimate.
	\end{proof}
	
	\section{The local construction}\label{sec:local-construction}
	
	The next lemma constructs a totally bounded metric on an interval.  Its
	completion is the compact component used later.  A probability measure on
	a sufficiently short interval will determine the distances from this
	component to subsequent components.  The distances within the component
	are placed below a given scale and are dominated by a prescribed measure
	near the origin.
	
	\begin{lemma}
		\label{lem:local}
		Let $R>0$, and let $\zeta$ be a nonzero atomless finite Borel measure on $\Rp$
		supported on $[0,R]$ such that $0\in\supp\zeta$.  Let $\varepsilon>0$.  There are
		$b>0$ and $\rho>0$ with the following property.
		
		For every compact interval $J\subseteq\R$ with $\diam J\le\rho$,
		every atomless Borel probability measure $\eta$ on $J$, and every
		$\beta\in(0,b]$, there is a Borel metric $d$ on $J$ such that
		$(J,d)$ is totally bounded,
		\[
		\diam_d J<\varepsilon
		\qquad\text{and}\qquad
		|x-y|\le d(x,y),\quad x,y\in J.
		\]
		If $U$ and $V$ are independent with common law $\eta$, then
		\[
		\beta\,\Law(d(U,V))\le\zeta.
		\]
	\end{lemma}
	\begin{proof}
		Since $\zeta$ is atomless and $0\in\supp\zeta$, every interval
		$(0,r]$, $r>0$, has positive $\zeta$-measure.  Since
		\[
		(0,r]
		=
		\bigcup_{k=0}^{\infty}(2^{-k-1}r,2^{-k}r],
		\]
		infinitely many of these dyadic intervals have positive measure.
		Choose one of them and write it as $(a,2a]$ for some $a>0$.  Since $\zeta$ is
		atomless, $\zeta([a,2a])=\zeta((a,2a])>0$.
		Thus, for every $r>0$, there is $a>0$ such that
		$[a,2a]\subseteq(0,r]$ and $\zeta([a,2a])>0$.  Put $r_0=\frac12\min\{\varepsilon,R\}$.
		By the preceding observation, we may choose $a_1>0$ such that
		\[
		[a_1,2a_1]\subseteq(0,r_0]
		\qquad\text{and}\qquad
		\zeta([a_1,2a_1])>0.
		\]
		Suppose that $a_n$ has been chosen.  Applying the same observation
		with $r=a_n/2$, choose $a_{n+1}>0$ such that
		\[
		[a_{n+1},2a_{n+1}]
		\subseteq(0,a_n/2]
		\qquad\text{and}\qquad
		\zeta([a_{n+1},2a_{n+1}])>0.
		\]
		This defines the sequence $(a_n)_{n\ge1}$ recursively.  The inclusions
		give
		\[
		2a_1<\min\{\varepsilon,R\},
		\qquad
		2a_{n+1}<a_n,
		\qquad
		\zeta([a_n,2a_n])>0,
		\quad n\ge1.
		\]
		In fact, $a_{n+1}\le a_n/4$, and hence $a_n\to0$.  For each $n\ge1$, partition
		$(a_n,2a_n]$ into finitely many half-open intervals of length at most
		$a_{n+1}$.  One of them has positive measure.  Let $I_n$ be its
		closure and put $m_n=\zeta(I_n)$.  Atomlessness of $\zeta$ shows that taking the
		closure does not change the measure.  Hence $m_n>0$ and
		$\diam I_n\le a_{n+1}$.  The intervals $I_n$ are pairwise disjoint
		because $2a_{n+1}<a_n$.  Set $b=m_1$ and $\rho=a_1$.
		
		Fix $J$, $\eta$, and $\beta$ as in the statement.  We construct nested
		finite partitions and the Borel maps used to define $d$.  Start with
		$\mathcal P_0=\{J\}$ and $s_0=1$.  For $n\ge1$, the family
		$\mathcal P_n$ will consist of pairwise disjoint Borel sets of positive
		$\eta$-measure whose union has full measure.  Every cell of
		$\mathcal P_n$ will be contained in a cell of
		$\mathcal P_{n-1}$.  After $\mathcal P_n$ has been chosen, a map
		$\Phi_{A,B}$ will be assigned to each unordered pair $\{A,B\}$ of
		distinct cells of $\mathcal P_n$ that are contained in the same cell of
		$\mathcal P_{n-1}$.  These maps will supply the distance between points
		first separated at level $n$.
		
		Suppose that the construction is complete through level $n-1$.  Form a
		finite common Borel refinement $\mathcal R_n$ of
		$\mathcal P_{n-1}$ in which every cell has Euclidean diameter at most
		$a_{n+1}$.  We also require that its image under each map constructed at
		a level $k<n$ have diameter at most $a_{n+1}$ whenever the cell lies in
		the domain of that map.  To obtain such a refinement, first use a finite
		interval partition of $J$ with mesh at most $a_{n+1}$.  For each of the
		finitely many earlier maps, refine further using its domain, the
		complement of its domain, and the inverse images of a finite interval
		partition of a bounded interval containing its range.  Choose this
		interval partition with mesh at most $a_{n+1}$.  Discard cells of zero
		$\eta$-measure.
		
		Write
		$s_{n-1}=\sum_{C\in\mathcal P_{n-1}}\eta(C)^2$ and choose $0<\delta_n<\min\{\frac{s_{n-1}}2,2^{-n},\frac{m_{n+1}}\beta\}$.
		By Lemma~\ref{lem:division}(ii), refine $\mathcal R_n$ to a family
		$\mathcal P_n$ in which every cell has mass at most $\delta_n$.
		Since $\eta$ is a probability measure,
		\[
		s_n=\sum_{C\in\mathcal P_n}\eta(C)^2
		\le\delta_n<s_{n-1}.
		\]
		As $\mathcal P_n$ refines $\mathcal R_n$, every
		$C\in\mathcal P_n$ satisfies
		\begin{equation}
			\label{eq:local-refinement}
			\diam C\le a_{n+1},\qquad
			\diam\Phi_{A,B}(C)\le a_{n+1}.
		\end{equation}
		The second inequality applies whenever $\Phi_{A,B}$ was constructed at
		a level $k<n$ and
		$C\subseteq\operatorname{dom}\Phi_{A,B}$. 
		
		Put $q_n=s_{n-1}-s_n>0$.  For $n=1$, we have
		$\beta q_1\le\beta\le m_1$.  For $n\ge2$, the choice of
		$\delta_{n-1}$ gives
		$\beta q_n\le\beta s_{n-1}\le\beta\delta_{n-1}<m_n$.
		Lemma~\ref{lem:division}(i) therefore gives a Borel set
		$E_n\subseteq I_n$ such that $\zeta(E_n)=\beta q_n$.  Define the
		atomless probability measure
		$\eta_n=\zeta|_{E_n}/(\beta q_n)$.

		For each unordered pair $\{A,B\}$ of distinct cells of
		$\mathcal P_n$ contained in the same cell of
		$\mathcal P_{n-1}$, choose $D_{A,B}\in\{A,B\}$.
		We claim that there is a Borel map
		$\Phi_{A,B}:D_{A,B}\to E_n$ such that
		\begin{equation}\label{eq:push-forward}
		(\Phi_{A,B})_\#
		\frac{\eta|_{D_{A,B}}}{\eta(D_{A,B})}
		=
		\eta_n.
		\end{equation}
		To prove this, write $D=D_{A,B}$ and let
		$\mu_D=\eta|_D/\eta(D)$.  Its distribution function
		\[
		F_D(x)=\mu_D(D\cap(-\infty,x]),\qquad x\in\R,
		\]
		is continuous because its jump at $x$ is
		$\mu_D(\{x\})=0$.  Thus the Borel function
		$T_D:D\to[0,1]$ given by $T_D(x)=F_D(x)$ sends $\mu_D$ to Lebesgue
		measure on $[0,1]$.  Indeed, the continuity of $F_D$ gives
		$\mathbb P(F_D(X)\le u)=u$ for $0\le u\le1$ whenever $X$ has law
		$\mu_D$.  A Borel generalised quantile function
		$Q_n:[0,1]\to\Rp$ for $\eta_n$ sends Lebesgue measure on $[0,1]$ to
		$\eta_n$.  The composition $Q_n\circ T_D$ therefore sends $\mu_D$ to
		$\eta_n$.  It belongs to $E_n$ outside a Borel $\mu_D$-null set because
		$\eta_n(E_n)=1$.  Changing its value on this set to a fixed point of
		$E_n$ gives the required map $\Phi_{A,B}$. Carrying out this construction for every such pair completes the
		induction step.
		
		Let
		\[
		\Omega=\bigcap_{n\ge0}\bigcup_{C\in\mathcal P_n}C,\qquad
		\mathcal P_n^\Omega
		=\{C\cap\Omega:C\in\mathcal P_n\}.
		\]
		The set $\Omega$ is Borel and $\eta(\Omega)=1$.  Each
		$C\cap\Omega$ has the same $\eta$-measure as $C$.  Thus the families
		$\mathcal P_n^\Omega$ are nested partitions of $\Omega$ into nonempty
		Borel sets.  Restrict each $\Phi_{A,B}$ to $D_{A,B}\cap\Omega$ and retain the same
		notation for the restricted map.  Since
		$\eta(D_{A,B}\setminus\Omega)=0$, we have
		$\eta(D_{A,B}\cap\Omega)=\eta(D_{A,B})$, and hence
		\[
		(\Phi_{A,B})_\#
		\frac{\eta|_{D_{A,B}\cap\Omega}}
		{\eta(D_{A,B}\cap\Omega)}
		=
		\eta_n.
		\]
		
		For distinct $x,y\in\Omega$, define $N(x,y)$ to be the least
		$n\ge1$ for which $x$ and $y$ lie in different cells of
		$\mathcal P_n^\Omega$.  This number is finite.  Indeed, if the points
		belonged to the same cell at every level, the first inequality in
		\eqref{eq:local-refinement} would give $|x-y|\le a_{n+1}$ for every
		$n\ge1$, and hence $x=y$.
		We set $N(x,y)=0$ for all remaining pairs $(x,y)\in J\times J$.  The
		resulting map $N:J\times J\to\{0,1,2,\ldots\}$ is Borel measurable.
		Indeed, for every $n\ge1$, the set
		$\{(x,y)\in J\times J:N(x,y)=n\}$ is a finite union of Borel
		rectangles, while the level set corresponding to $n=0$ is the
		complement of the union of these sets over all $n\ge1$.
		
		For distinct $x,y\in\Omega$, let $n=N(x,y)$, and let $A$ and $B$ be the cells of
		$\mathcal P_n$ whose intersections with $\Omega$ contain $x$ and $y$,
		respectively.  Define $d(x,x)=0$ and
		\[
		d(x,y)=d(y,x)=
		\begin{cases}
			\Phi_{A,B}(x),&D_{A,B}=A,\\
			\Phi_{A,B}(y),&D_{A,B}=B.
		\end{cases}
		\]
		Then $d(x,y)\in E_n\subseteq[a_n,2a_n]$, so $d$ is symmetric and
		positive away from the diagonal.  We also have
		$|x-y|\le d(x,y)$.  If $n=1$, this follows from
		$\diam J\le a_1$.  If $n\ge2$, the points lie in a common cell at
		level $n-1$, whose Euclidean diameter is at most $a_n$ by
		\eqref{eq:local-refinement}.
		
		We verify the triangle inequality on $\Omega$.  Triangles with a
		repeated point need no further argument.  For three distinct points,
		the minimum of their three first-separation levels is attained at least
		twice.  At the first level at which two of the points are separated,
		either one point is in a different cell from the other two or all three
		points are in different cells.  In both cases at least two pairs are
		separated at that level.  After relabelling, we may assume that
		\[
		N(x,y)=N(y,z)=n\le m=N(x,z).
		\]
		If $m=n$, all three distances belong to $[a_n,2a_n]$, so each is at
		most the sum of the other two.
		
		Suppose that $m>n$.  At level $n$, the points $x$ and $z$ lie in one
		cell and $y$ lies in another cell with the same parent.  If the map for
		this pair is defined on the cell containing $y$, then
		$d(x,y)=d(z,y)$.  Otherwise the map is evaluated at $x$ and $z$.
		If $m=n+1$, its range lies in $E_n\subseteq I_n$, whose diameter is at
		most $a_{n+1}=a_m$.  If $m\ge n+2$, then $x$ and $z$ lie in a common
		cell at level $m-1$.  The second inequality in
		\eqref{eq:local-refinement}, applied at that level, again bounds the
		difference of the two values by $a_m$.  In both cases,
		\[
		|d(x,y)-d(z,y)|\le a_m\le d(x,z).
		\]
		Figure~\ref{fig:local-construction} illustrates the case
		$m\ge n+2$ in which the map is evaluated at $x$ and $z$.  This proves
		the two triangle inequalities having $d(x,y)$ or $d(z,y)$ on the left.
		The remaining one follows from
		\[
		d(x,z)\le2a_m\le2a_n\le d(x,y)+d(y,z).
		\]
		
		\begin{figure}[t]
			\centering
			\begin{tikzpicture}[
				every node/.style={font=\small},
				cell/.style={
					draw,
					rounded corners=1pt,
					minimum height=7mm,
					minimum width=15mm,
					inner sep=2pt
				},
				selected/.style={cell,fill=black!8},
				map/.style={-{Stealth[length=2mm]},semithick}
				]
				\node[cell] (P) at (-1.2,4.2) {$P$};
				\node[cell,minimum width=18mm] (B) at (-2.8,2.9) {$B\ni y$};
				\node[selected,minimum width=25mm] (A) at (0.4,2.9)
				{$A=D_{A,B}$};
				\node[selected,minimum width=24mm] (C) at (0.4,1.45)
				{$C\ni x,z$};
				\node[cell,minimum width=19mm] (X) at (-0.65,0)
				{$C_x\ni x$};
				\node[cell,minimum width=19mm] (Z) at (1.45,0)
				{$C_z\ni z$};
				
				\draw (P.east) .. controls (0.15,4.05) and (1.45,3.75) .. (A.north east);
				\draw (P.west) .. controls (-2.55,4.05) and (-3.85,3.75) .. (B.north west);
				\draw[densely dashed] (A) --
				node[right,font=\scriptsize] {$n+1,\ldots,m-1$} (C);
				\draw (C.west) .. controls (-1.25,1.2) and (-1.65,0.75) .. (X.north west);
				\draw (C.east) .. controls (2.05,1.2) and (2.45,0.75) .. (Z.north east);
				
				\node[anchor=east] at (-4.55,4.2) {$\mathcal P_{n-1}$};
				\node[anchor=east] at (-4.55,2.9) {$\mathcal P_n$};
				\node[anchor=east] at (-4.55,1.45) {$\mathcal P_{m-1}$};
				\node[anchor=east] at (-4.55,0) {$\mathcal P_m$};
				
				\node[font=\scriptsize] at (-1.15,3.45)
				{$N(x,y)=N(z,y)=n$};
				\node[font=\scriptsize] at (0.4,0.55)
				{$N(x,z)=m$};
				
				\draw (3.15,1.85) -- (7.45,1.85);
				\draw (3.15,1.72) -- (3.15,1.98);
				\draw (7.45,1.72) -- (7.45,1.98);
				
				\node[below=2pt] at (3.15,1.72) {$a_n$};
				\node[below=2pt] at (7.45,1.72) {$2a_n$};
				
				\draw[line width=2.5pt,black!35]
				(3.85,1.85) -- (6.75,1.85);
				\node[above=3pt] at (5.3,1.85) {$E_n$};
				
				\draw[line width=4pt,black!70]
				(4.85,1.85) -- (5.75,1.85);
				
				\fill (5.05,1.85) circle (1.5pt);
				\fill (5.57,1.85) circle (1.5pt);
				
				\node at (4.65,2.55) {$\Phi(x)$};
				\node at (6.02,2.55) {$\Phi(z)$};
				
				\draw[thin] (4.82,2.35) -- (5.05,1.9);
				\draw[thin] (5.86,2.35) -- (5.57,1.9);
				
				\draw[
				decorate,
				decoration={brace,amplitude=4pt,mirror}
				]
				(4.85,1.25) -- (5.75,1.25)
				node[midway,below=5pt]
				{$\diam\Phi(C)\le a_m$};
				
				\draw[map] (A.east) to[out=15,in=155]
				node[above,sloped] {$\Phi=\Phi_{A,B}$}
				(4.15,2.02);
				
				\node at (5.45,3.25)
				{$\Phi_\#(\eta|_A/\eta(A))=\eta_n$};
				
				\node at (1.45,-1.15)
				{$|d(x,y)-d(z,y)|
					=|\Phi(x)-\Phi(z)|
					\le a_m
					\le d(x,z)$};
			\end{tikzpicture}
			\caption{The case $N(x,y)=N(z,y)=n<m=N(x,z)$ with $m\ge n+2$.
				The map $\Phi_{A,B}$ is defined on $A$ and is evaluated at $x$ and
				$z$.  Since $x$ and $z$ lie in the same cell
				$C\in\mathcal P_{m-1}$, the second inequality in
				\eqref{eq:local-refinement} gives
				$\diam\Phi_{A,B}(C)\le a_m$.  Moreover,
				$d(x,z)\in E_m\subseteq[a_m,2a_m]$.}
			\label{fig:local-construction}
		\end{figure}
		
		If two points of $\Omega$ lie in the same cell of
		$\mathcal P_n^\Omega$, their first separation level is at least
		$n+1$, and their distance is therefore at most $2a_{n+1}$.  Given
		$\delta>0$, choose $n$ such that $2a_{n+1}<\delta$ and choose one point
		from each cell of $\mathcal P_n^\Omega$.  These points form a finite
		$\delta$-net for $\Omega$.  Hence $(\Omega,d)$ is totally bounded.
		
		It remains to define $d$ on $J\setminus\Omega$.  Use the Euclidean
		distance on this set and put $d(x,y)=2a_1$ when exactly one of $x$ and
		$y$ belongs to $\Omega$.  In every triangle meeting both sets, two
		distances equal $2a_1$ and the remaining distance is at most $2a_1$.
		Thus the triangle inequality continues to hold.  We also have
		$|x-y|\le d(x,y)$ and
		$\diam(J,d)\le2a_1<\varepsilon$.  The set $J\setminus\Omega$ is totally bounded under the Euclidean
		metric because it is contained in the compact interval $J$.  Given
		$\delta>0$, the union of a finite $\delta$-net for $\Omega$ and a
		finite $\delta$-net for $J\setminus\Omega$ is a finite $\delta$-net
		for $(J,d)$.  Hence $(J,d)$ is totally bounded.
		
		The metric is Borel.  The product $J\times J$ is the disjoint union of
		$(J\setminus\Omega)^2$, $\Omega\times(J\setminus\Omega)$,
		$(J\setminus\Omega)\times\Omega$, the diagonal in
		$\Omega\times\Omega$, and the rectangles $(A\cap\Omega)\times(B\cap\Omega)$ and $(B\cap\Omega)\times(A\cap\Omega)$,
		where $n\ge1$ and $A,B\in\mathcal P_n$ are distinct cells contained
		in the same cell of $\mathcal P_{n-1}$.  There are only countably many
		such Borel sets.  On each of them, $d$ is the Euclidean distance, the
		constant $2a_1$, zero, or the composition of a coordinate projection
		with one of the Borel maps $\Phi_{A,B}$.  Hence
		$d:J\times J\to\Rp$ is Borel measurable.
		
		Let $U$ and $V$ be independent with common law $\eta$.  Almost surely,
		$U,V\in\Omega$ and $U\ne V$.  The probability that $U$ and $V$ lie in the same cell of
		$\mathcal P_n^\Omega$ is $s_n=\sum_{C\in\mathcal P_n}\eta(C)^2$.
		The events that $U$ and $V$ lie in the same cell at levels $n-1$ and
		$n$ are nested.  Hence, by the definition of $q_n$,
		\[
		\mathbb P(N(U,V)=n)=s_{n-1}-s_n=q_n.
		\]
		Since $s_n\le\delta_n<2^{-n}$, we have $s_n\to0$.  For every $K\ge1$, $\sum_{n=1}^Kq_n=s_0-s_K$.
		Consequently,
		\[
		\sum_{n=1}^{\infty}q_n
		=
		\lim_{K\to\infty}(s_0-s_K)
		=
		1.
		\]
		
		By \eqref{eq:push-forward} and the restriction to the full-measure set
		$\Omega$, the measure $\eta_n$ is the conditional distribution of $d(U,V)$
		given $N(U,V)=n$.  To verify this, let
		$g:\Rp\to\R$ be a bounded Borel function.  An unordered pair
		$\{A,B\}$ of distinct cells of $\mathcal P_n$ with a common parent
		contributes
		\[
		2\eta(A)\eta(B)
		\int_{\Rp}g(t)\,\eta_n(dt)
		\]
		to
		$\mathbb E[g(d(U,V))\mathbf 1_{\{N(U,V)=n\}}]$.  For a fixed cell
		$P\in\mathcal P_{n-1}$, the sum of the coefficients
		$2\eta(A)\eta(B)$ over its distinct level-$n$ children is $\eta(P)^2-\sum_{\substack{A\in\mathcal P_n\\A\subseteq P}}\eta(A)^2$.
		Summing over $P\in\mathcal P_{n-1}$ therefore gives
		\[
		\begin{aligned}
			\mathbb E\big[
			g(d(U,V))\mathbf 1_{\{N(U,V)=n\}}
			\big]
			&=
			\big(
			s_{n-1}-s_n
			\big)
			\int_{\Rp}g(t)\,\eta_n(dt)
			=
			q_n\int_{\Rp}g(t)\,\eta_n(dt).
		\end{aligned}
		\]
		This proves the asserted conditional distribution.  Consequently,
		\[
		\beta\,\Law(d(U,V))
		=
		\sum_{n=1}^{\infty}\beta q_n\eta_n
		=
		\sum_{n=1}^{\infty}\zeta|_{E_n}
		\le\zeta.
		\]
		The final inequality holds because $E_n\subseteq I_n$ for every
		$n\ge1$ and the intervals $I_n$, $n\ge1$, are pairwise disjoint.
	\end{proof}
	
The later constructions require compact components.  The metric space
$(J,d)$ obtained in Lemma~\ref{lem:local} is totally bounded but need
not be complete, so we take its metric completion.  The measure $\eta$
is transported to the completion by the canonical isometric embedding
of $J$.  Since this embedding is isometric, the sampled distance law is
unchanged.

\begin{corollary}
	\label{cor:compact-block}
	Let $R,\zeta,\varepsilon,b,\rho$ be as in
	Lemma~\ref{lem:local}, and fix $J,\eta,\beta$ satisfying the
	conditions in its conclusion. Then there exist a compact metric probability space
	$(X,d_X,\mu)$ and a $1$-Lipschitz map $\phi:X\to J$ such that
	\[
	\phi_\#\mu=\eta,\qquad
	\diam X<\varepsilon,\qquad
	\beta\,\Law(d_X(U,V))\le\zeta.
	\]
	Here $U$ and $V$ are independent with common law $\mu$.
\end{corollary}

\begin{proof}
	Let $d$ be the metric provided by Lemma~\ref{lem:local}.  Let
	$(X,d_X)$ be the completion of $(J,d)$ and let
	$\iota:J\to X$ be the canonical isometric embedding.  Since $(J,d)$
	is totally bounded, its completion $X$ is compact.
	
	The map $\iota$ is Borel measurable with respect to the original
	Borel structure on $J$.  Indeed, every $d$-ball is a Euclidean Borel
	subset of $J$ because $d:J\times J\to\Rp$ is Borel measurable.
	Moreover, $(J,d)$ is separable, so its topology has a countable basis
	consisting of $d$-balls.  It follows that every $d$-open subset of
	$J$ is Euclidean Borel.  Since $\iota$ is continuous with respect to
	the $d$-topology, it is Borel measurable with respect to the
	Euclidean Borel structure on $J$.
	
	The identity map from $(J,d)$ to the Euclidean interval $J$ is
	$1$-Lipschitz because $|x-y|\le d(x,y)$ for all $x,y\in J$.
	Since the Euclidean interval $J$ is complete, this map extends
	uniquely to a $1$-Lipschitz map $\phi:X\to J$.  In particular,
	$\phi\circ\iota$ is the identity on $J$.
	
	Put $\mu=\iota_\#\eta$.  Then $\mu$ is a Borel probability measure on
	$X$, and $\phi_\#\mu=\eta$.  Since $\iota(J)$ is dense in $X$, taking
	the completion does not change the diameter.  Hence
	$\diam X=\diam(J,d)<\varepsilon$.
	
	Finally, the isometry property of $\iota$ gives
	\[
	(d_X)_\#(\mu\otimes\mu)
	=
	(d_X\circ(\iota\times\iota))_\#(\eta\otimes\eta)
	=
	d_\#(\eta\otimes\eta).
	\]
	Thus completion does not change the distance law.  Together with
	Lemma~\ref{lem:local}, the preceding identity gives
	$\beta\,\Law(d_X(U,V))\le\zeta$.
\end{proof}
	
We shall also use the following scaling consequence.

\begin{corollary}
	\label{cor:homogeneity}
	Let $R$, $\varepsilon$, and $\zeta$ be as in
	Lemma~\ref{lem:local}, and let $b,\rho$ have the property stated there.
	For every $c>0$, the constants $cb,\rho$ have the same property with
	$\zeta$ replaced by $c\zeta$.
\end{corollary}

\begin{proof}
	Fix $c>0$.  Let $J\subseteq\R$ be a compact interval, let $\eta$ be an
	atomless Borel probability measure on $J$, and suppose that
	$\diam J\le\rho$ and $\beta\in(0,cb]$.  Then
	$\beta/c\in(0,b]$.  The property of $b,\rho$ gives a metric $d$
	satisfying the conclusions of Lemma~\ref{lem:local} and
	\[
	\frac{\beta}{c}\,\Law(d(U,V))\le\zeta,
	\]
	where $U$ and $V$ are independent with $\Law(U)=\Law(V)=\eta$.
	Multiplying this inequality by $c$ gives
	$\beta\,\Law(d(U,V))\le c\zeta$.
\end{proof}

	\section{The compactly supported case}\label{sec:compact-case}

	We first use the compact components provided by
	Corollary~\ref{cor:compact-block} to treat one interval of distances.
	Suppose that the target measure is decomposed as
	$\Theta=\zeta+\eta$, where $\zeta$ is supported on $[0,R]$ and $\eta$
	on $[R,2R]$.  The measure $\eta$ will be realised by distances between
	different compact components.  The distances within these components
	are dominated by part of $\zeta$, and the remainder of $\zeta$ is
	retained for the next step.  Figure~\ref{fig:ordered-gluing}
	illustrates the construction.
	
	\begin{lemma}
		\label{lem:annulus-transfer}
		Let $\zeta$ and $\eta$ be finite Borel measures on $\Rp$ such that
		$\Theta=\zeta+\eta$ is an atomless probability measure and
		\[
		\supp\zeta\subseteq[0,R],\qquad
		0\in\supp\zeta,\qquad
		\supp\eta\subseteq[R,2R]
		\]
		for some $R>0$.  Then there are an integer $N\ge0$, compact metric
		probability spaces $(X_i,d_i,\mu_i)$, $i=1,\ldots,N$, positive
		numbers $p_1,\ldots,p_N,p_\ast$, and $1$-Lipschitz maps
		$\phi_i:X_i\to[R,2R]$ such that $\diam X_i<R$, $i=1,\ldots,N$ and $p_\ast+\sum_{i=1}^N p_i=1$.
		For independent $U_i,V_i\sim\mu_i$, let
		$\nu_i=\Law(d_i(U_i,V_i))$ and set $\tau=\sum_{i=1}^N p_i^2\nu_i$.
		Then $\tau\le\frac12\zeta$.  Moreover, the atomless measure
		$\zeta'=\zeta-\tau$ satisfies
		\[
		\supp\zeta'\subseteq[0,R],\qquad
		\zeta'(\Rp)=p_\ast^2,\qquad
		\zeta'\ge\frac12\zeta,\qquad
		0\in\supp\zeta'.
		\]
		Finally,
		\[
		\eta=\sum_{i=1}^N c_i\eta_i,\qquad
		\eta_i=(\phi_i)_\#\mu_i,\qquad
		c_i=2p_i\Big(
		p_\ast+\sum_{j=i+1}^N p_j
		\Big).
		\]
		If $\zeta'/p_\ast^2$ has a compact realisation of diameter at most
		$R$, then $\Theta$ has a compact realisation.
	\end{lemma}

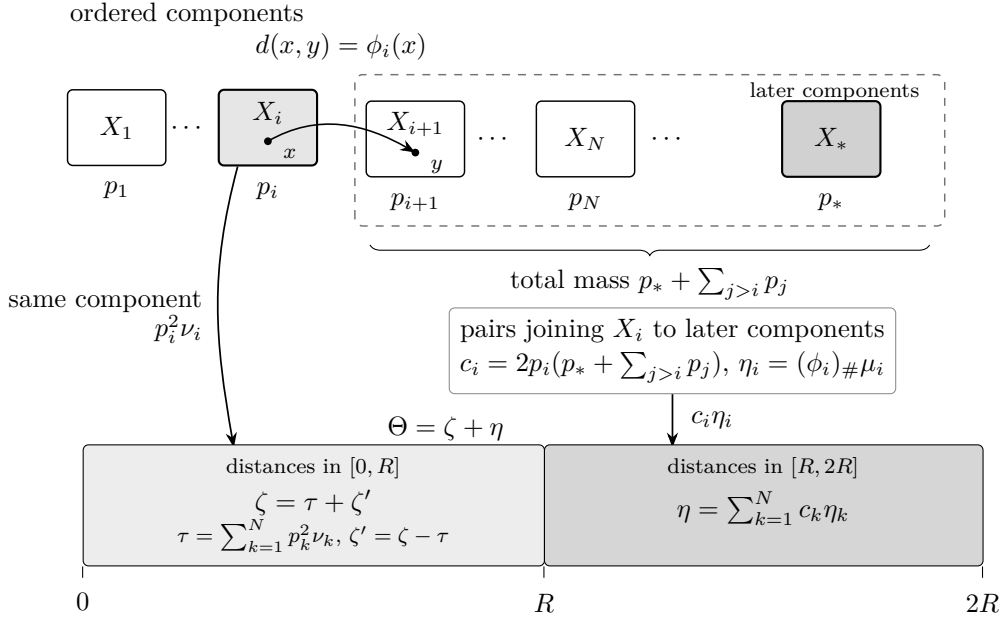
\begin{figure}[t]
	\centering
	\begin{tikzpicture}[
		every node/.style={font=\small},
		component/.style={
			draw,
			rounded corners=2pt,
			minimum width=13mm,
			minimum height=10mm,
			fill=white,
			line width=0.55pt
		},
		active/.style={
			component,
			fill=black!10,
			line width=0.8pt
		},
		terminal/.style={
			component,
			fill=black!18,
			line width=0.8pt
		},
		group/.style={
			draw=black!55,
			dashed,
			rounded corners=3pt,
			line width=0.55pt
		},
		map/.style={
			-{Stealth[length=2mm]},
			semithick
		},
		contribution/.style={
			-{Stealth[length=2mm]},
			line width=0.65pt
		},
		formula/.style={
			draw=black!45,
			rounded corners=2pt,
			fill=white,
			inner sep=4pt
		},
		lower/.style={
			draw,
			rounded corners=2pt,
			fill=black!7,
			minimum height=16mm
		},
		upper/.style={
			draw,
			rounded corners=2pt,
			fill=black!16,
			minimum height=16mm
		}
		]
		\node[anchor=west,font=\small] at (0.05,7.55)
		{ordered components};
		\node[font=\small] at (5.15,2.08)
		{$\Theta=\zeta+\eta$};
		
		\node[
		component,
		label={[label distance=2pt]below:$p_1$}
		] (xone) at (0.8,6.05) {$X_1$};
		
		\node at (1.75,6.05) {$\cdots$};
		
		\node[
		active,
		label={[label distance=2pt]below:$p_i$}
		] (xi) at (2.8,6.05) {};
		
		\node at (2.8,6.25) {$X_i$};
		
		\draw[group] (3.95,4.75) rectangle (11.75,6.75);
		
		\node[anchor=east,font=\scriptsize] at (11.55,6.53)
		{later components};
		
		\node[
		component,
		label={[label distance=2pt]below:$p_{i+1}$}
		] (xip) at (4.75,5.90) {};
		
		\node at (4.75,6.10) {$X_{i+1}$};
		\node at (5.8,5.90) {$\cdots$};
		
		\node[
		component,
		label={[label distance=2pt]below:$p_N$}
		] (xn) at (7.0,5.90) {$X_N$};
		
		\node at (8.1,5.90) {$\cdots$};
		
		\node[
		terminal,
		label={[label distance=2pt]below:$p_\ast$}
		] (xstar) at (10.25,5.90) {$X_\ast$};
		
		\path (xi.center) ++(0,-0.18) coordinate (xpoint);
		\path (xip.center) ++(0,-0.18) coordinate (ypoint);
		
		\fill (xpoint) circle (1.45pt);
		\node[
		below right=-1pt and 2pt,
		font=\scriptsize
		] at (xpoint) {$x$};
		
		\fill (ypoint) circle (1.45pt);
		\node[
		below right=-1pt and 2pt,
		font=\scriptsize
		] at (ypoint) {$y$};
		
		\draw[map] (xpoint) to[out=30,in=150] (ypoint);
		
		\node[fill=white,inner sep=1.5pt] at (3.77,7.12)
		{$d(x,y)=\phi_i(x)$};
		
		\draw[
		decorate,
		decoration={
			brace,
			amplitude=4pt,
			mirror
		}
		]
		(4.15,4.50) -- (11.55,4.50)
		node[midway,below=5pt,align=center]
		{total mass
			$p_\ast+\textstyle\sum_{j>i}p_j$};
		
		\draw[contribution]
		([xshift=-4mm]xi.south)
		to[out=-105,in=100]
		node[pos=0.54,left=2pt,align=right]
		{same component\\[-1pt]$p_i^2\nu_i$}
		(2.35,1.82);
		
		\node[formula,align=center] (coeff) at (8.15,3.10)
		{pairs joining $X_i$ to later components\\[1pt]
			$c_i=2p_i
			(p_\ast+\textstyle\sum_{j>i}p_j)$,
			$\eta_i=(\phi_i)_\#\mu_i$};
		
		\draw[contribution] (coeff.south) --
		node[pos=0.48,right=3pt] {$c_i\eta_i$}
		(8.15,1.82);
		
		\node[
		lower,
		minimum width=61mm,
		anchor=south west
		] (low) at (0.35,0.25) {};
		
		\node[
		upper,
		minimum width=58mm,
		anchor=south west
		] (high) at (6.45,0.25) {};
		
		\node[font=\scriptsize] at (3.4,1.56)
		{distances in $[0,R]$};
		
		\node at (3.4,1.10)
		{$\zeta=\tau+\zeta'$};
		
		\node[font=\scriptsize] at (3.4,0.66)
		{$\tau=\textstyle\sum_{k=1}^N p_k^2\nu_k$,
			$\zeta'=\zeta-\tau$};
		
		\node[font=\scriptsize] at (9.35,1.56)
		{distances in $[R,2R]$};
		
		\node at (9.35,1.02)
		{$\eta=\textstyle\sum_{k=1}^N c_k\eta_k$};
		
		\draw (0.35,0.25) -- (0.35,0.10);
		\draw (6.45,0.25) -- (6.45,0.10);
		\draw (12.25,0.25) -- (12.25,0.10);
		
		\node[below=2pt] at (0.35,0.10) {$0$};
		\node[below=2pt] at (6.45,0.10) {$R$};
		\node[below=2pt] at (12.25,0.10) {$2R$};
	\end{tikzpicture}
	
	\caption{The ordered gluing used in
		Lemma~\ref{lem:annulus-transfer}. For $x\in X_i$, the distance
		from $x$ to every point in a later component is $\phi_i(x)$.
		Pairs contained in $X_i$ contribute $p_i^2\nu_i$ to the measure
		below $R$, while pairs joining $X_i$ to a later component
		contribute $c_i\eta_i$ on $[R,2R]$, where
		$c_i=2p_i(p_\ast+\sum_{j>i}p_j)$ and
		$\eta_i=(\phi_i)_\#\mu_i$. If $\zeta'/p_\ast^2$ is realised by a terminal component $X_\ast$,
		then this component contributes $\zeta'$.}
	\label{fig:ordered-gluing}
\end{figure}
	
\begin{proof}
		Let $\alpha=\zeta(\Rp)$, so that $\eta(\Rp)=1-\alpha$.  Since
		$0\le\zeta,\eta\le\Theta$, both $\zeta$ and $\eta$ are atomless by
		Lemma~\ref{lem:division}(iii).  Since $0\in\supp\zeta$, we have
		$\alpha>0$.  If $\eta=0$, take $N=0$, $p_\ast=1$, $\tau=0$, and
		$\zeta'=\zeta$. The conclusions concerning
	the constructed measures are then immediate.  The final assertion is
	precisely its assumption in this case.  We may therefore suppose that
	$\eta\ne0$, and hence that $\alpha\in(0,1)$.
	
	Apply Lemma~\ref{lem:local} to $\zeta$ with $\varepsilon=R$, and denote
	the resulting constants by $b$ and $\rho$.  Choose $\kappa>0$ such that
	\[
	\kappa\le
	\min\left\{\gamma_0(\alpha),
	\frac{\alpha b}{2(1-\alpha)}\right\},
	\]
	where $\gamma_0(\alpha)$ is given by
	Lemma~\ref{lem:mass-selection}.
	
	Partition $[R,2R]$ into finitely many intervals of length at most
	$\rho$.  Consider one such interval of positive $\eta$-mass $m$ and put
	$q=\lceil m/\kappa\rceil$.  Since $\eta$ is atomless, the distribution
	function of its restriction to this interval is continuous.  It
	therefore divides the interval into $q$ consecutive  intervals of
	equal mass $m/q\le\kappa$.  Applying this construction to every
	positive-mass interval and discarding the intervals of zero mass gives
	finitely many intervals $I_1,\ldots,I_N$.
	
	Put $J_i=\overline{I_i}$, $c_i=\eta(I_i)$, and
	$\eta_i=\eta|_{I_i}/c_i$.  Since endpoints have zero $\eta$-mass, we
	may regard $\eta_i$ as an atomless probability measure on $J_i$.  The
	construction gives
	\[
	\eta=\sum_{i=1}^N c_i\eta_i,\qquad
	0<c_i\le\kappa,\qquad
	J_i\subseteq[R,2R],\qquad
	\diam J_i\le\rho.
	\]
	
	Apply Lemma~\ref{lem:mass-selection} to the sequence
	$c_1,\ldots,c_N$ in reverse order.  This gives the weights
	$p_1,\ldots,p_N,p_\ast$ and the identities stated in the lemma.  Its
	estimate yields
	\[
	\delta=\sum_{i=1}^N p_i^2
	\le\frac{(1-\alpha)\kappa}{\alpha}
	\le\frac b2.
	\]
	
	For $i=1,\ldots,N$, set
	$\zeta_i=p_i^2\zeta/(2\delta)$.  Since $\delta>0$, these measures are
	well defined, and
	$\sum_{i=1}^N\zeta_i=\zeta/2$.  By
	Corollary~\ref{cor:homogeneity}, the constants associated with
	$\zeta_i$ are $p_i^2b/(2\delta)$ and $\rho$.  The inequality
	$\delta\le b/2$ implies
	$p_i^2\le p_i^2b/(2\delta)$.  We may therefore apply
	Corollary~\ref{cor:compact-block} with coefficient $p_i^2$.  It gives
	a compact metric probability space $(X_i,d_i,\mu_i)$ and a
	$1$-Lipschitz map $\phi_i:X_i\to J_i$ such that
	\[
	(\phi_i)_\#\mu_i=\eta_i,\qquad
	\diam X_i<R,\qquad
	p_i^2\nu_i\le\zeta_i.
	\]
	Summing the last measure inequalities gives
	$\tau\le\sum_{i=1}^N\zeta_i=\zeta/2$.  Hence
	$\zeta'=\zeta-\tau$ is nonnegative and satisfies
	$\zeta'\ge\zeta/2$.  Since $\zeta'\le\zeta$, we also have
	$\supp\zeta'\subseteq[0,R]$, and
	Lemma~\ref{lem:division}(iii) shows that $\zeta'$ is atomless.
	Moreover, every neighbourhood of the origin has positive
	$\zeta'$-mass because it has positive $\zeta$-mass.  Thus
	$0\in\supp\zeta'$.  Finally, the mass identity in
	Lemma~\ref{lem:mass-selection} gives
	$\zeta'(\Rp)=\alpha-\delta=p_\ast^2$.
	
	Assume now that $\zeta'/p_\ast^2$ has a compact realisation
	$(X_\ast,d_\ast,\mu_\ast)$ of diameter at most $R$.  On the disjoint
	union $X=X_1\sqcup\ldots\sqcup X_N\sqcup X_\ast$
	retain the internal metrics and set
	$d(x,y)=d(y,x)=\phi_i(x)$ whenever $x\in X_i$ and $y$ belongs to a
	later component.  Lemma~\ref{lem:annular-amalgamation} shows that $d$
	is a metric.
	Distances between distinct components are at least $R$.  Each
	component is therefore open and Borel in $(X,d)$.  Consequently,
	$\mu=\sum_{i=1}^N p_i\mu_i+p_\ast\mu_\ast$ is a Borel probability
	measure.  Since $X$ is a finite union of compact components, it is
	compact.
	
	Let $U$ and $V$ be independent with common law $\mu$, and let
	$\nu_\ast$ denote the internal distance law of $X_\ast$. Then
	$p_\ast^2\nu_\ast=\zeta'$, and
	Lemma~\ref{lem:distance-identity}(i) gives
	\[
	\Law(d(U,V))
	=\zeta'
	+\sum_{i=1}^N p_i^2\nu_i
	+\sum_{i=1}^N c_i\eta_i
	=\zeta'+\tau+\eta
	=\Theta.
	\]
	This proves the final assertion.
\end{proof}
	
	\begin{remark}
		The construction of the components and of $\zeta'$ does not assume that
		$\zeta'/p_\ast^2$ has already been realised.  In the following proof this
		measure is carried to smaller scales, and all components are assembled
		only after the iteration.  
	\end{remark}
	
	We now prove the compactly supported case of
	Theorem~\ref{thm:main}.  This result will also be used in the general
	case to realise the bounded part that remains after the unbounded tail
	has been treated.  For a measure supported on $[0,A]$, the proof applies
	Lemma~\ref{lem:annulus-transfer} successively on the intervals
	$[2^{-n}A,2^{-n+1}A]$.
	
	\begin{theorem}
		\label{thm:compact}
		Let $\Theta$ be an atomless probability measure supported on $[0,A]$
		for some $A>0$, and assume that $0\in\supp\Theta$.  Then $\Theta$ is
		realised by a compact metric probability space of diameter at most $A$.
	\end{theorem}
	
	\begin{proof}
		The proof proceeds from larger to smaller distance scales.  At each
		step, Lemma~\ref{lem:annulus-transfer} is applied to the next interval.
		It produces finitely many compact components and leaves a measure
		supported closer to the origin for the following step.  When the
		components are inserted into the final space, their masses are
		rescaled so that the total mass sums to one.
		
		Put $R_n=2^{-n}A$, $n=0,1,2,\ldots$.  We construct finite measures $\Lambda_n$
		recursively.  Start with $\Lambda_0=\Theta$, $M_0=1$, and $P_0=1$.
		Suppose that $\Lambda_{n-1}$ is atomless, supported on
		$[0,R_{n-1}]$, and has the origin in its support.  Let
		\[
		M_{n-1}=\Lambda_{n-1}(\Rp),\qquad
		\theta_{n-1}=\frac{\Lambda_{n-1}}{M_{n-1}},
		\]
		and split
		\[
		\theta_{n-1}=\zeta_n+\eta_n,\qquad
		\zeta_n=\theta_{n-1}|_{[0,R_n]},\qquad
		\eta_n=\theta_{n-1}|_{(R_n,R_{n-1}]}.
		\]
		The measure $\zeta_n$ is nonzero and still has the origin in its
		support.
		
		If $\eta_n=0$, put $N_n=0$, $p_{n,\ast}=1$, $\tau_n=0$, and
		$\zeta_n'=\zeta_n$.  Otherwise, Lemma~\ref{lem:annulus-transfer}
		produces an integer $N_n\ge1$, compact metric probability spaces
		$(X_{n,i},d_{n,i},\mu_{n,i})$, $i=1,\ldots,N_n$, maps
		$\phi_{n,i}:X_{n,i}\to[R_n,R_{n-1}]$, and positive weights
		$p_{n,1},\ldots,p_{n,N_n},p_{n,\ast}$.  In either case,
		\[
		p_{n,\ast}+\sum_{i=1}^{N_n}p_{n,i}=1.
		\]
		For $i=1,\ldots,N_n$, let
		$\eta_{n,i}=(\phi_{n,i})_\#\mu_{n,i}$ and
		$\nu_{n,i}=\Law(d_{n,i}(U_{n,i},V_{n,i}))$, where
		$U_{n,i},V_{n,i}\sim\mu_{n,i}$ are independent.  The conclusions of Lemma~\ref{lem:annulus-transfer}, together with the
		conventions for the case $\eta_n=0$, give
		\[
		\eta_n=\sum_{i=1}^{N_n}c_{n,i}\eta_{n,i},\qquad
		c_{n,i}
		=2p_{n,i}\Big(
		p_{n,\ast}+\sum_{j=i+1}^{N_n}p_{n,j}
		\Big),
		\]
		and
		\[
		\tau_n=\sum_{i=1}^{N_n}p_{n,i}^2\nu_{n,i},\qquad
		\zeta_n'=\zeta_n-\tau_n\ge\frac12\zeta_n,\qquad
		\zeta_n'(\Rp)=p_{n,\ast}^2.
		\]
		
		Define
		\[
		\Lambda_n=M_{n-1}\zeta_n',\qquad
		M_n=\Lambda_n(\Rp)=M_{n-1}p_{n,\ast}^2.
		\]
		The measure $\Lambda_n$ is atomless by
		Lemma~\ref{lem:division}(iii), is supported on $[0,R_n]$, and has the
		origin in its support.  Moreover, induction gives
		\[
		\Lambda_n
		=M_{n-1}(\zeta_n-\tau_n)
		\le M_{n-1}\zeta_n
		=\Lambda_{n-1}|_{[0,R_n]}
		\le\Theta|_{[0,R_n]}.
		\]
		Since $\Theta(\{0\})=0$, continuity from above yields
		$M_n\le\Theta([0,R_n])\to0$ as $n\to\infty$.
		
		Set $P_n=\sqrt{M_n}$ and
		$w_{n,i}=P_{n-1}p_{n,i}$ for $i=1,\ldots,N_n$.  Then
		$P_n=P_{n-1}p_{n,\ast}$ and
		\begin{equation}\label{eq:RekursionVonPn}
		P_n+\sum_{i=1}^{N_n}w_{n,i}=P_{n-1}.
		\end{equation}
		Summing \eqref{eq:RekursionVonPn} over $n=1,\ldots,K$ gives
		\[
		\sum_{n=1}^K\sum_{i=1}^{N_n}w_{n,i}=P_0-P_K=1-P_K.
		\]
		Since $P_K\to0$ as $K\to\infty$, it follows that
		$\sum_{n=1}^{\infty}\sum_{i=1}^{N_n}w_{n,i}=1$.
		
		We now assemble all components.  Add a point $\omega$ which does not
		belong to any of them and which will serve as their common limit as
		$n\to\infty$.  Let
		\[
		X=\{\omega\}\sqcup
		\bigsqcup_{n=1}^{\infty}
		\bigsqcup_{i=1}^{N_n}X_{n,i}.
		\]
		Order the components first by $n$ and then by $i$, and place $\omega$
		after all of them.  Recall that, for $n\ge1$ and
		$i=1,\ldots,N_n$, the map
		$\phi_{n,i}:X_{n,i}\to[R_n,R_{n-1}]$ is $1$-Lipschitz.  Retain the
		internal metrics and set $d(\omega,\omega)=0$.  If $x\in X_{n,i}$ and
		$y$ belongs to a later component or equals $\omega$, set
		\[
		d(x,y)=d(y,x)=\phi_{n,i}(x).
		\]
		
		We verify that $d$ is a metric.  Suppose first that $x,x'$ belong to
		the same component and $y$ belongs to a later component.  The reverse
		triangle inequality follows from the fact that $\phi_{n,i}$ is
		$1$-Lipschitz.  The remaining inequality follows from
		\[
		d_{n,i}(x,x')\le\diam X_{n,i}\le R_n
		\le d(x,y)+d(x',y).
		\]
		If $y$ belongs to an earlier component, its distances to $x$ and $x'$
		are equal, and the same diameter estimate applies.
		
		Now consider three points belonging to three different sets among the
		components and $\{\omega\}$, and let $x\in X_{n,i}$ belong to the
		earliest one.  The two distances from $x$ are equal and lie in
		$[R_n,R_{n-1}]$.  The distance between the two later points is
		determined by the earlier of them and is at most
		$R_{n-1}=2R_n$.  Hence all triangle inequalities hold.
		
		Each component $X_{n,i}$ is open, since its distance from its
		complement is at least $R_n>0$, and is therefore Borel. The singleton
		$\{\omega\}$ is also Borel.  Thus
		\[
		\mu=\sum_{n=1}^{\infty}
		\sum_{i=1}^{N_n}w_{n,i}\mu_{n,i},
		\qquad
		\mu(\{\omega\})=0,
		\]
		is a Borel probability measure.
		
		The space $(X,d)$ is totally bounded.  Given $\varepsilon>0$, choose $n_0$
		such that $R_{n_0-1}<\varepsilon$.  Every component of level
		$n\ge n_0$ lies in $B(\omega,\varepsilon)$.  The components with
		$n<n_0$ form a finite union of compact spaces and are therefore
		totally bounded.
		
		The space is also complete.  Let $(x_k)_{k\ge1}$ be a Cauchy sequence in $X$.
		If infinitely many terms are equal to $\omega$, then the corresponding
		constant subsequence converges to $\omega$, and the Cauchy property
		implies convergence of the entire sequence.  Otherwise, discard the
		finitely many terms equal to $\omega$ and denote the level of $x_k$ by
		$n_k$.  If $n_k\to\infty$, then $d(x_k,\omega)\le R_{n_k-1}\to 0$.
		If the levels do not tend to infinity, there is an $L\ge1$ such that
		infinitely many terms belong to levels $1,\ldots,L$.  These levels
		contain only finitely many components, so one compact component
		contains a subsequence.  This subsequence has a convergent further
		subsequence, and the Cauchy property implies convergence of the entire
		sequence to the same limit.  Thus $X$ is complete and, being totally
		bounded, compact.
		
		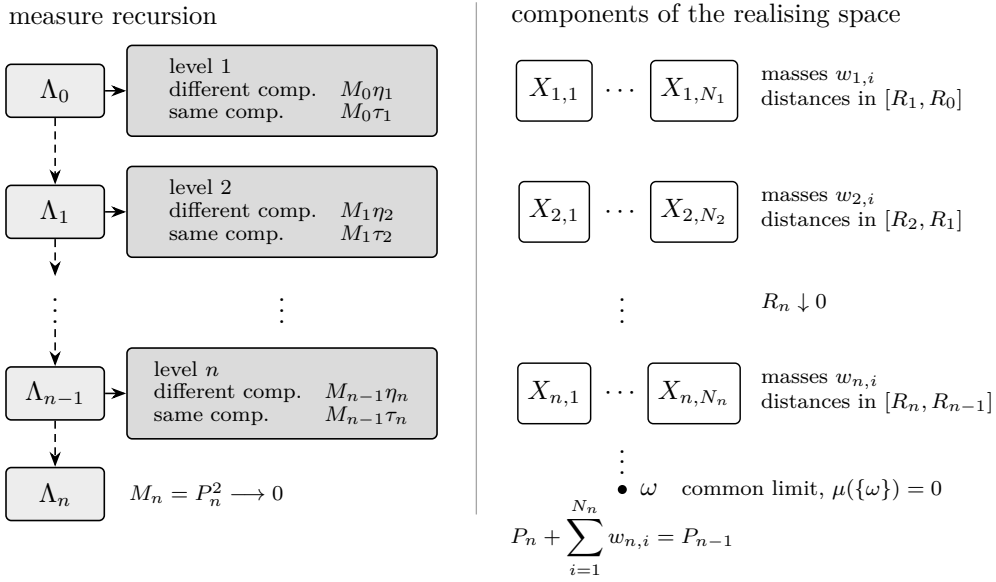
\begin{figure}[t]
			\centering
			\begin{tikzpicture}[
				every node/.style={font=\small},
				residual/.style={
					draw,
					rounded corners=2pt,
					minimum width=13mm,
					minimum height=7mm,
					fill=black!7,
					line width=0.6pt
				},
				contribution/.style={
					draw,
					rounded corners=2pt,
					minimum width=41mm,
					minimum height=12mm,
					fill=black!14,
					line width=0.6pt,
					align=center,
					font=\scriptsize
				},
				component/.style={
					draw,
					rounded corners=2pt,
					minimum width=10mm,
					minimum height=8mm,
					fill=white,
					line width=0.55pt
				},
				arrow/.style={-{Stealth[length=2mm]},semithick},
				carry/.style={
					-{Stealth[length=2mm]},
					densely dashed,
					line width=0.6pt
				},
				scalelabel/.style={
					font=\scriptsize,
					anchor=west,
					align=left
				}
				]
				\node[anchor=west] at (0,7.15) {measure recursion};
				\node[anchor=west] at (6.65,7.15)
				{components of the realising space};
				\draw[black!35,line width=0.5pt]
				(6.32,0.55) -- (6.32,7.32);
				
				\node[residual] (lambda0) at (0.75,6.15) {$\Lambda_0$};
				\node[contribution] (cont1) at (3.75,6.15)
				{\begin{tabular}{@{}l@{\quad}l@{}}
					level $1$        & \\[-1pt]
					different comp. & $M_0\eta_1$ \\[-1pt]
					same comp.      & $M_0\tau_1$
				\end{tabular}};
				\draw[arrow] (lambda0.east) -- (cont1.west);
				
				\node[residual] (lambda1) at (0.75,4.55) {$\Lambda_1$};
				\node[contribution] (cont2) at (3.75,4.55)
				{\begin{tabular}{@{}l@{\quad}l@{}}
					level $2$        & \\[-1pt]
					different comp. & $M_1\eta_2$ \\[-1pt]
					same comp.      & $M_1\tau_2$
				\end{tabular}};
				\draw[carry] (lambda0.south) -- (lambda1.north);
				\draw[arrow] (lambda1.east) -- (cont2.west);
				
				\node at (0.75,3.35) {$\vdots$};
				\node at (3.75,3.35) {$\vdots$};
				
				\node[residual] (lambdanminus) at (0.75,2.15)
				{$\Lambda_{n-1}$};
				\node[contribution] (contn) at (3.75,2.15)
				{\begin{tabular}{@{}l@{\quad}l@{}}
					level $n$        & \\[-1pt]
					different comp. & $M_{n-1}\eta_n$ \\[-1pt]
					same comp.      & $M_{n-1}\tau_n$
				\end{tabular}};
				\draw[carry] (lambda1.south) -- (0.75,3.70);
				\draw[carry] (0.75,3.00) -- (lambdanminus.north);
				\draw[arrow] (lambdanminus.east) -- (contn.west);
				
				\node[residual] (lambdan) at (0.75,0.82) {$\Lambda_n$};
				\draw[carry] (lambdanminus.south) -- (lambdan.north);
				\node[font=\scriptsize] at (2.75,0.82)
				{$M_n=P_n^2\longrightarrow0$};
				
				\node[component] at (7.35,6.15) {$X_{1,1}$};
				\node at (8.25,6.15) {$\cdots$};
				\node[component] at (9.20,6.15) {$X_{1,N_1}$};
				\node[scalelabel] at (9.95,6.15)
				{masses $w_{1,i}$\\distances in $[R_1,R_0]$};
				
				\node[component,minimum width=9mm] at (7.35,4.55)
				{$X_{2,1}$};
				\node at (8.25,4.55) {$\cdots$};
				\node[component,minimum width=9mm] at (9.20,4.55)
				{$X_{2,N_2}$};
				\node[scalelabel] at (9.95,4.55)
				{masses $w_{2,i}$\\distances in $[R_2,R_1]$};
				
				\node at (8.25,3.35) {$\vdots$};
				\node[scalelabel] at (9.95,3.35) {$R_n\downarrow0$};
				
				\node[component,minimum width=8mm] at (7.35,2.15)
				{$X_{n,1}$};
				\node at (8.25,2.15) {$\cdots$};
				\node[component,minimum width=8mm] at (9.20,2.15)
				{$X_{n,N_n}$};
				\node[scalelabel] at (9.95,2.15)
				{masses $w_{n,i}$\\distances in $[R_n,R_{n-1}]$};
				
				\node at (8.25,1.35) {$\vdots$};
				\fill (8.25,0.88) circle (1.7pt);
				\node[right=3pt] at (8.25,0.88) {$\omega$};
				\node[font=\scriptsize,anchor=west] at (8.90,0.88)
				{common limit, $\mu(\{\omega\})=0$};
				
				\node[font=\scriptsize,align=center] at (8.25,0.25)
				{$\displaystyle
					P_n+\sum_{i=1}^{N_n}w_{n,i}=P_{n-1}$};
			\end{tikzpicture}
			\caption{The iteration in the proof of
				Theorem~\ref{thm:compact}.  At level $n$, pairs in different
				components whose earlier component has level $n$ contribute
				$M_{n-1}\eta_n$, while pairs within the level-$n$ components
				contribute $M_{n-1}\tau_n$.  The remaining measure $\Lambda_n$
				is carried to the next scale.  The components of level $n$ have
				masses $w_{n,i}$, and their distances to later components lie in
				$[R_n,R_{n-1}]$.  Since $R_n\to0$, the levels accumulate at
				the point $\omega$.}
			\label{fig:compact-iteration}
		\end{figure}
		
		It remains to identify the distance law.  Let $U$ and $V$ be
		independent with common law $\mu$.  Since
		$\Lambda_{n-1}=M_{n-1}\theta_{n-1}$,
		$\theta_{n-1}=\zeta_n+\eta_n$, and
		$\zeta_n=\zeta_n'+\tau_n$, we have
		\[
		\Lambda_{n-1}
		=M_{n-1}\eta_n+M_{n-1}\tau_n+\Lambda_n.
		\]
		Figure~\ref{fig:compact-iteration} summarises the measure recursion and
		the corresponding arrangement of the components.
		Iterating this identity and using
		$\Lambda_n(\Rp)=M_n\to0$ as $n\to\infty$ yields
		\begin{equation}\label{eq:KompaktesMassZerlegung}
		\Theta
		=\sum_{n=1}^{\infty}M_{n-1}\eta_n
		+\sum_{n=1}^{\infty}M_{n-1}\tau_n.
		\end{equation}
		The internal contribution of the level-$n$ components is
		\[
		\sum_{i=1}^{N_n}w_{n,i}^2\nu_{n,i}
		=M_{n-1}\tau_n.
		\]
		The same telescoping argument shows that the total mass of all
		components at levels $m>n$ is $P_n$.  Hence, for a fixed component
		$X_{n,i}$, the total mass in all later components is
		\[
		\begin{aligned}
			\sum_{j=i+1}^{N_n}w_{n,j}
			+\sum_{m=n+1}^{\infty}\sum_{j=1}^{N_m}w_{m,j}
			&=P_{n-1}\sum_{j=i+1}^{N_n}p_{n,j}+P_n=P_{n-1}\Big(
			p_{n,\ast}+\sum_{j=i+1}^{N_n}p_{n,j}
			\Big),
		\end{aligned}
		\]
		where we used \eqref{eq:RekursionVonPn}.
		Enumerate the nonempty components $X_{n,i}$ in lexicographic order.
		Their weights sum to one, so Lemma~\ref{lem:distance-identity}(i),
		with no terminal component, applies to this ordered family.  The point
		$\omega$ has $\mu$-measure zero and therefore does not contribute to
		the distance law.
		Lemma~\ref{lem:distance-identity}(i), applied with this ordering,
		therefore shows that the contribution from pairs in different
		components whose earlier component has level $n$ is
		\[
		M_{n-1}\sum_{i=1}^{N_n}
		2p_{n,i}\Big(
		p_{n,\ast}+\sum_{j=i+1}^{N_n}p_{n,j}
		\Big)\eta_{n,i}
		=M_{n-1}\eta_n.
		\]
		The internal contributions
		and those from pairs in different components therefore agree with the
		two terms in \eqref{eq:KompaktesMassZerlegung}.  Hence
		$\Law(d(U,V))=\Theta$. 
		
		Finally, every internal distance in a component of level $n$ is at
		most $R_n\le A$.  If the earlier of two distinct components has level
		$n$, then their mutual distance belongs to
		$[R_n,R_{n-1}]\subseteq[0,A]$.  The same holds when the later point is
		$\omega$.  Hence $\diam X\le A$.
	\end{proof}
	
\section{The general case}\label{sec:general-case}

It remains to consider atomless laws with unbounded support.  We first
divide the part outside a bounded interval into probability measures
supported on short, ordered intervals.  The corresponding components
will be placed after a compact core in the same order.  The fact that
these intervals move to infinity will ensure completeness of the
resulting space.

\begin{lemma}
	\label{lem:tail-partition}
	Fix $A>0$, and let $\eta$ be a finite atomless Borel measure on $(A,\infty)$ with
	unbounded support.  Given $\rho,\gamma>0$, there are pairwise disjoint
	intervals $B_i$, $i=1,2,\ldots$, which cover $(A,\infty)$ up to an
	$\eta$-null set and whose closures $J_i=\overline{B_i}$ in $\R$ satisfy
	\[
	0<\eta(B_i)\le\gamma,\qquad
	\diam J_i\le\rho,\qquad
	\sup J_i\le\inf J_j\quad\text{for }1\le i<j,
	\]
	and $\inf J_i\to\infty$ as $i\to\infty$.
\end{lemma}

\begin{proof}
	Partition $(A,\infty)$ into the consecutive half-open intervals
	\[
	H_k=(A+k\rho/2,A+(k+1)\rho/2],
	\qquad k=0,1,\ldots.
	\]
	Their endpoints have zero $\eta$-measure.  If
	$m_k=\eta(H_k)>0$, put $N_k=\lceil m_k/\gamma\rceil$.  Write
	$H_k=(a_k,b_k]$ and set $t_{k,0}=a_k$ and
	$t_{k,N_k}=b_k$.  Since $\eta$ is atomless, the distribution
	function of $\eta|_{H_k}$ is continuous.  Hence, for
	$\ell=1,\ldots,N_k-1$, there is a point $t_{k,\ell}$ such that
	\[
	\eta\bigl(H_k\cap(-\infty,t_{k,\ell}]\bigr)
	=\frac{\ell m_k}{N_k}.
	\]
	The intervals
	$(t_{k,\ell-1},t_{k,\ell}]$, $\ell=1,\ldots,N_k$, are consecutive,
	have mass $m_k/N_k\le\gamma$, and cover $H_k$.  Thus each $H_k$
	produces only finitely many intervals.  Discard the intervals $H_k$
	of zero mass and enumerate all remaining subintervals in increasing
	order as $B_1,B_2,\ldots$.
	
	If $B_i\subseteq H_k$, then
	$J_i\subseteq\overline{H_k}$ and hence
	$\diam J_i\le\rho/2\le\rho$.  The chosen enumeration gives
	$\sup J_i\le\inf J_j$ whenever $i<j$.  Since $\eta$ has unbounded
	support, intervals $H_k$ of positive mass occur for arbitrarily large
	$k$.  Each $H_k$ contributes only finitely many intervals, so the
	index $k$ for which $B_i\subseteq H_k$ tends to infinity with $i$.
	Consequently, $\inf J_i\to\infty$.
\end{proof}

\begin{proof}[Proof of Theorem~\ref{thm:main}]
	Necessity of $0\in\supp\Theta$ was proved in the introduction.  We
	prove sufficiency.  If $\Theta$ has bounded support, the claim follows
	from Theorem~\ref{thm:compact}, so we can assume that the support of $\Theta$
	is unbounded.
	
	Fix $A>0$ and write
	\[
	\zeta=\Theta|_{[0,A]},\qquad
	\eta=\Theta|_{(A,\infty)},\qquad
	\alpha=\zeta(\Rp).
	\]
	Since $0\le\zeta,\eta\le\Theta$, both $\zeta$ and $\eta$ are atomless
	by Lemma~\ref{lem:division}(iii).
	The support assumptions imply that $0<\alpha<1$.  Apply
	Lemma~\ref{lem:local} to $\zeta$ with $\varepsilon=A$, and denote the
	resulting constants by $b$ and $\rho$.
	
	Choose $\gamma>0$ such that
	\[
	\gamma\le
	\min\Big\{\gamma_0(\alpha),
	\frac{\alpha b}{2(1-\alpha)}\Big\}.
	\]
	Apply Lemma~\ref{lem:tail-partition} to $\eta$ with parameters
	$\rho$ and $\gamma$, and let $B_i$, $i=1,2,\ldots$, be the intervals
	provided by the lemma.  Write $J_i=\overline{B_i}$ and put
	\[
	c_i=\eta(B_i),\qquad
	\eta_i=\frac{\eta|_{B_i}}{c_i},
	\qquad i=1,2,\ldots.
	\] Since the endpoints of $B_i$ have zero
	$\eta$-measure, we may regard $\eta_i$ as an atomless probability
	measure on the compact interval $J_i$.  Moreover,
	\[
	\eta=\sum_{i=1}^{\infty}c_i\eta_i.
	\]
	Lemma~\ref{lem:mass-selection} gives positive numbers
	$p_0,p_1,p_2,\ldots$ satisfying
	\[
	p_0+\sum_{i=1}^{\infty}p_i=1,\qquad
	c_i=2p_i\Big(
	p_0+\sum_{j=1}^{i-1}p_j
	\Big),\qquad
	p_0^2+\sum_{i=1}^{\infty}p_i^2=\alpha.
	\]
	Its estimate gives
	\[
	\delta:=\sum_{i=1}^{\infty}p_i^2
	\le\frac{(1-\alpha)\gamma}{\alpha}
	\le\frac b2.
	\]
	
	For $i=1,2,\ldots$, put
	$\zeta_i=p_i^2\zeta/(2\delta)$.  Then
	$\sum_{i=1}^{\infty}\zeta_i=\zeta/2$.  By
	Corollary~\ref{cor:homogeneity}, Lemma~\ref{lem:local} applies to
	$\zeta_i$ with $p_i^2b/(2\delta)$ in place of $b$ and with $\rho$
	unchanged.  Since $\delta\le b/2$, we have $p_i^2\le\frac{p_i^2b}{2\delta}$.
	Corollary~\ref{cor:compact-block} therefore gives a compact metric
	probability space
	$(X_i,d_i,\mu_i)$ and a $1$-Lipschitz map
	$\phi_i:X_i\to J_i$ such that
	\[
	(\phi_i)_\#\mu_i=\eta_i,\qquad
	\diam X_i<A,\qquad
	p_i^2\nu_i\le\zeta_i,
	\]
	where $\nu_i$ is the internal distance law of $X_i$.  Consequently,
	\[
	\tau:=\sum_{i=1}^{\infty}p_i^2\nu_i
	\le\sum_{i=1}^{\infty}\zeta_i
	=\frac12\zeta.
	\]
	
	Put $\zeta'=\zeta-\tau$.  Since $0\le\zeta'\le\zeta$,
	Lemma~\ref{lem:division}(iii) shows that $\zeta'$ is atomless.
	Moreover, $\zeta'$ is supported on $[0,A]$ and satisfies
	$\zeta'\ge\zeta/2$.  It therefore has the origin in its support.
	The mass identity from Lemma~\ref{lem:mass-selection} gives
	$\zeta'(\Rp)=\alpha-\delta=p_0^2$.  By
	Theorem~\ref{thm:compact}, there is a compact metric probability
	space $(X_0,d_0,\mu_0)$ of diameter at most $A$ whose distance law is
	$\zeta'/p_0^2$.
	
	Set
	\[
	X=X_0\sqcup\bigsqcup_{i=1}^{\infty}X_i,\qquad
	\mu=p_0\mu_0+\sum_{i=1}^{\infty}p_i\mu_i,
	\]
	and retain the internal metrics.  For $i\ge1$, $x\in X_i$, and
	$y\in\bigsqcup_{j=0}^{i-1}X_j$, define
	\[
	d(x,y)=d(y,x)=\phi_i(x).
	\]
	
	We verify that $d$ is a metric on $X$.  Suppose first that two points
	$x,x'$ belong to the same component $X_i$ with $i\ge1$.  If the third
	point belongs to an earlier component, the reverse triangle
	inequality follows from the fact that $\phi_i$ is $1$-Lipschitz.  If
	it belongs to a later component, the two cross-distances are equal.
	The remaining inequality follows in both cases from
	\[
	d_i(x,x')\le\diam X_i\le A
	\le d(x,y)+d(x',y),
	\]
	because every distance between distinct components is at least $A$.
	The same argument applies when $x,x'\in X_0$.
	
	Now consider points $x_i\in X_i$, $x_j\in X_j$, and $x_k\in X_k$
	with $i>j>k\ge0$.  The two distances from $x_i$ are equal to
	$\phi_i(x_i)$, while
	\[
	d(x_j,x_k)=\phi_j(x_j)
	\le\sup J_j
	\le\inf J_i
	\le\phi_i(x_i).
	\]
	Hence all triangle inequalities hold.
	
	Every component is at distance at least $A$ from its complement and
	is therefore open and Borel.  It follows that $\mu$ is a Borel
	probability measure.  The space is separable because the union of
	countable dense subsets of its components is countable and dense in
	$X$.
	
	To prove completeness, let $(x_n)_{n\ge1}$ be a Cauchy sequence in $X$.
	Choose $n_0$ such that $d(x_n,x_m)<1$ for all $n,m\ge n_0$, and let
	$r\ge0$ be the index of the component containing $x_{n_0}$.  Since
	$\inf J_i\to\infty$ as $i\to\infty$, there is an $N>r$ such that
	$\inf J_i>1$ for every $i\ge N$.  No term $x_n$ with $n\ge n_0$
	can belong to such a component $X_i$, since its distance from
	$x_{n_0}$ would be greater than $1$.  The tail of the sequence is
	therefore contained in the finite union
	$X_0\sqcup\ldots\sqcup X_{N-1}$.  Distinct components are at distance
	at least $A>0$, so the sequence is eventually contained in one of
	them.  It converges there because every component is compact.
	
	Finally, let $U$ and $V$ be independent with common law $\mu$.
	Lemma~\ref{lem:distance-identity}(ii) gives
	\[
	\begin{aligned}
		\Law(d(U,V))
		&=\zeta'
		+\sum_{i=1}^{\infty}p_i^2\nu_i
		+\sum_{i=1}^{\infty}
		2p_i\Big(
		p_0+\sum_{j=1}^{i-1}p_j
		\Big)\eta_i\\
		&=\zeta'+\tau
		+\sum_{i=1}^{\infty}c_i\eta_i\\
		&=\zeta+\eta
		=\Theta.
	\end{aligned}
	\]
	This proves sufficiency and completes the proof.
\end{proof}

\begin{proof}[Proof of Corollary~\ref{cor:abc}]
	Every absolutely continuous probability measure is atomless, so the
	claim follows from Theorem~\ref{thm:main}.
\end{proof}

\section{Geometry of the realising spaces}
\label{sec:geometry}

The preceding constructions also determine the topology of the
realising spaces and show that their metrics may be chosen arbitrarily
close to ultrametrics in the multiplicative sense made precise below.
Two metrics $d$ and $u$ on the same set $S$ are called
$L$-bi-Lipschitz equivalent, where $L\geq1$, if
\[
L^{-1}u(x,y)\leq d(x,y)\leq Lu(x,y),
\qquad x,y\in S.
\]
Recall that a topological space is zero-dimensional if it has a basis
of sets that are both open and closed, and perfect if it has no
isolated points.  Every ultrametric space is zero-dimensional.  A
metric space is proper if every closed bounded subset is compact.  We
refer to \cite[Chapter~1]{BuragoBuragoIvanov2001} for the metric space
terminology.  The Cantor space is the product space
$\{0,1\}^{\mathbb N}$.  By Brouwer's characterisation theorem
\cite[Theorem~7.4]{Kechris1995}, every nonempty perfect compact
metrizable zero-dimensional space is homeomorphic to the Cantor space.

\begin{corollary}
	\label{cor:geometry}
	Let $\Theta$ be an atomless Borel probability measure on $\Rp$ such
	that $0\in\supp\Theta$.  For every $L>1$, there is a realisation
	$(S,d,\mu)$ of $\Theta$ such that $\supp\mu=S$ and $d$ is
	$L$-bi-Lipschitz equivalent to an ultrametric $u$ on $S$.  In fact,
	\[
	u(x,y)\leq d(x,y)\leq Lu(x,y),
	\qquad x,y\in S.
	\]
	If $\Theta$ has bounded support, then $S$ may be chosen compact and
	homeomorphic to the Cantor space.  In this case,
	\[
	\{d(x,y):x,y\in S\}=\supp\Theta,
	\qquad
	\diam S=\max\supp\Theta.
	\]
	If $\Theta$ has unbounded support, then $S$ may be chosen proper and
	homeomorphic to a countable disjoint union of Cantor spaces.
\end{corollary}

\begin{proof}
	Fix $L>1$.  We first refine the local construction.  Consider one
	application of Lemma~\ref{lem:local}, and retain the notation
	$a_n$, $I_n$, $\Omega$, and $N$ from its proof.  The sequence $(a_n)_{n\ge1}$ may additionally be chosen so that
	$a_{n+1}\le(L-1)a_n$ for every $n\ge1$.  Indeed, after $a_n$ has
	been chosen, apply the observation at the beginning of the proof of
	Lemma~\ref{lem:local} with $r=\min\{a_n/2,2(L-1)a_n\}$.
	The resulting number $a_{n+1}$ satisfies
	$[a_{n+1},2a_{n+1}]\subseteq(0,r]$.  Hence
	$2a_{n+1}<a_n$ and $a_{n+1}\le(L-1)a_n$, as required.
	
	Put $r_n=\inf I_n$.  For distinct $x,y\in\Omega$, define
	$u(x,y)=r_{N(x,y)}$, and put $u(x,x)=0$.  The ultrametric
	inequality is immediate if two of the three points coincide.  For
	three distinct points $x,y,z\in\Omega$, the nested partitions give $N(x,z)\ge\min\{N(x,y),N(y,z)\}$.
	Moreover, $r_{n+1}\le2a_{n+1}<a_n\le r_n$, and hence
	$u$ is an ultrametric. If $N(x,y)=n$, then $d(x,y)\in I_n$.
	Since $\diam I_n\le a_{n+1}$ and $r_n\ge a_n$,
	\[
	u(x,y)=r_n\le d(x,y)
	\le r_n+a_{n+1}\le Lr_n=L u(x,y).
	\]
	Since $u\le d\le Lu$ on $\Omega$, the identity map between
	$(\Omega,d)$ and $(\Omega,u)$ extends to a bi-Lipschitz
	identification of their completions, and the same inequalities hold
	on the completions.
	
	Let $(X,d_X)$, $\iota:J\to X$, and $\mu=\iota_\#\eta$ be as in the
	proof of Corollary~\ref{cor:compact-block}.  We claim that $\supp\mu=\overline{\iota(\Omega)}$,
	where the closure is taken in $X$.  For $x\in\Omega$, let $C_n(x)$
	be the cell of $\mathcal P_n^\Omega$ containing $x$.  This cell has
	positive $\eta$-measure.  Moreover, any two points of $C_n(x)$ have
	first separation level at least $n+1$, and hence
	\[
	\diam_d C_n(x)\le2a_{n+1}.
	\]
	Since $a_n\to0$ as $n\to\infty$, every $d_X$-neighbourhood of $\iota(x)$ has positive
	$\mu$-measure.  Thus $\iota(\Omega)\subseteq\supp\mu$.  Since $\supp\mu$ is closed,
	$\overline{\iota(\Omega)}\subseteq\supp\mu$.  Conversely, if
	$z\notin\overline{\iota(\Omega)}$, there is an open neighbourhood $G$
	of $z$ disjoint from $\iota(\Omega)$.  Then
	$\iota^{-1}(G)\subseteq J\setminus\Omega$, and therefore
	$\mu(G)=\eta(\iota^{-1}(G))=0$.  Hence $z\notin\supp\mu$, which proves
	the claim.
	
	Consequently, the completion of $(\Omega,d)$ is precisely
	$\supp\mu$.  The ultrametric constructed on $\Omega$ therefore
	extends to $\supp\mu$, where it still satisfies
	\[
	u(x,y)\le d_X(x,y)\le Lu(x,y).
	\]
	Thus every compact component used below may be restricted to the
	support of its measure and equipped with an ultrametric satisfying
	the same comparison.
	
	Next, we repeat the proof of Theorem~\ref{thm:compact}, retaining its
	notation.  In each application of
	Lemma~\ref{lem:annulus-transfer} on $[R_n,R_{n-1}]$, partition this
	interval into intervals of length at most
	$\min\{\rho,(L-1)R_n\}$, where $\rho$ is the constant used in that
	application.  After the subdivisions by mass, denote the resulting
	closed intervals by $J_{n,1},\ldots,J_{n,N_n}$ and enumerate them in
	decreasing order.  Put $h_{n,i}=\inf J_{n,i}$.  Then
	\[
	h_{n,i}\le t\le Lh_{n,i},
	\qquad t\in J_{n,i}.
	\]
	
	Let $u_{n,i}$ be the ultrametric within $X_{n,i}$ obtained above.
	Retain these ultrametrics within the components and put
	$u(\omega,\omega)=0$.  If $x\in X_{n,i}$ and $y$ belongs to a later
	component or equals $\omega$, set
	$u(x,y)=u(y,x)=h_{n,i}$.  The numbers $h_{n,i}$ are
	nonincreasing in the order of the components, and
	$\diam_{u_{n,i}}X_{n,i}<R_n\le h_{n,i}$.  Hence $u$ is an
	ultrametric on the assembled space, and $u\le d\le Lu$.
	
	Replace this space by the support $S$ of its measure.  This does not
	change the distance law.  The measure is atomless and has full
	support on $S$, so $S$ has no isolated points.  The metrics $d$ and
	$u$ induce the same topology, which is zero-dimensional.  In the
	bounded case, $S$ is compact and hence homeomorphic to the Cantor
	space.
	
	If $S$ is compact and $\mu$ has full support, then every
	$d(x,y)$, $x,y\in S$, belongs to the support of the distance law.
	Indeed, continuity of $d$ and full support give positive measure to
	every neighbourhood of this value.  The converse inclusion follows
	because $\{d(x,y):x,y\in S\}$ is compact.  Therefore
	\[
	\{d(x,y):x,y\in S\}=\supp\Theta,
	\]
	which also gives $\diam S=\max\supp\Theta$.
	
	For the unbounded case, repeat the proof of
	Theorem~\ref{thm:main}.  Choose $A>0$, put
	$\zeta=\Theta|_{[0,A]}$, and let $\rho$ be obtained by applying
	Lemma~\ref{lem:local} to $\zeta$ with $\varepsilon=A$.  Apply
	Lemma~\ref{lem:tail-partition} with
	$\rho_L=\min\{\rho,(L-1)A\}$.  Retain the notation $J_i$ and $X_i$
	from that proof, and put $h_i=\inf J_i$.  Then
	$\sup J_i\le Lh_i$, and $(h_i)_{i\ge1}$ is nondecreasing.
	
	Equip the compact core $X_0$ and every $X_i$, $i\ge1$, with the
	ultrametrics obtained above.  If $x\in X_i$ and $y$ belongs to an
	earlier component, set $u(x,y)=u(y,x)=h_i$.  The ordering shows that
	$u$ is an ultrametric and that $u\le d\le Lu$.
	
	After restriction to the measure supports, every component is
	homeomorphic to the Cantor space.  Since the components are open and
	closed, their union is homeomorphic to a countable disjoint union of
	Cantor spaces.  If $x_0\in X_0$, then
	$d(x,x_0)\ge\inf J_i\to\infty$ for $x\in X_i$.  Hence every bounded
	subset meets only finitely many compact components.  Every closed
	bounded subset is therefore compact, so $S$ is proper.
\end{proof}

The restriction $L>1$ cannot be replaced by $L=1$.  As noted in the
introduction, a separable ultrametric space has only countably many
positive distance values.  Its distance law is
therefore supported on a countable set and cannot be atomless.  Thus
the realising metric may be chosen arbitrarily close to an ultrametric
in the bi-Lipschitz sense, but cannot itself be an ultrametric.

\medspace

The results and the constructions in the present paper leave several natural questions for future research.

\begin{enumerate}[label=\textnormal{(\roman*)}]
	\item Which probability measures with atoms are feasible?  The support
	condition is not sufficient in general
	\cite[Proposition~2.2]{AldousBlancCurien2025}.
	
	\item Which atomless laws admit a realisation on a connected or
	geodesic metric space?  If $\mu$ has full support and $S$ is
	connected, then $\supp\Theta$ must be an interval, since it is the
	closure of the continuous image $d(S\times S)$.
	
	\item Which atomless laws can be realised by distances in a Hilbert
	space?  This is the setting of the original question in
	\cite{AldousMO2022}.
	
	\item For a fixed $n\geq3$, which probability measures on
	$\Rp^{\binom{n}{2}}$ can occur as the joint law of $\bigl(d(X_i,X_j)\bigr)_{1\leq i<j\leq n}$
	for independent points $X_1,\ldots,X_n$ with common law?  Such laws
	must satisfy the triangle inequalities, be invariant under
	relabelling the points, and admit compatible extensions to larger
	sample sizes.  The case $n=3$ was proposed in
	\cite{AldousBlancCurien2025}.
	
	\item The spaces constructed above are zero-dimensional in the
	topological sense, but our argument does not give any bound on
	their Hausdorff dimension. It is therefore natural to ask under which assumptions on $\Theta$ the realising space can be chosen of finite Hausdorff dimension.
\end{enumerate}

	\section*{Acknowledgement}
	
The authors have been supported by the DFG project \textit{Limit theorems for the volume of random projections of $\ell_p$-balls} (project number 516672205). CT was also supported by the DFG Priority Program SPP 2265 \textit{Random Geometric Systems}. 

During the preparation of this work, the authors used Claude (Anthropic)
and ChatGPT (OpenAI) to assist in exploring and checking mathematical
arguments and in editing the manuscript.  All mathematical statements
and proofs were independently verified and finalised by the authors.

	\bigskip
	
	\noindent
	Christoph Th\"ale and Philipp Tuchel\\
	Faculty of Mathematics, Ruhr University Bochum, Germany\\
	Email addresses: \texttt{christoph.thaele@rub.de},
	\texttt{philipp.tuchel@rub.de}
	

\begin{thebibliography}{99}
		
		\bibitem{AldousMO2022}
		D.~J. Aldous,
		\newblock Distributions of distance between two random points in Hilbert
		space,
		\newblock MathOverflow, Question 428539 (2022),
		\newblock \url{https://mathoverflow.net/q/428539}.
		
		\bibitem{AldousBlancCurien2025}
		D.~J. Aldous, G.~Blanc, and N.~Curien,
		\newblock The distance problem on measured metric spaces,
		\newblock \emph{Ann. Fac. Sci. Toulouse Math.} \textbf{34} (2025),
		1345--1365,
		\newblock \url{https://doi.org/10.5802/afst.1835}.
		
		\bibitem{BuragoBuragoIvanov2001}
		D.~Burago, Y.~Burago, and S.~Ivanov,
		\newblock \emph{A Course in Metric Geometry},
		\newblock Graduate Studies in Mathematics, vol.~33,
		American Mathematical Society, Providence, RI, 2001,
		\newblock \url{https://doi.org/10.1090/gsm/033}.
		
		\bibitem{DovgosheyShcherbak2022}
		O.~Dovgoshey and V.~Shcherbak,
		\newblock The range of ultrametrics, compactness, and separability,
		\newblock \emph{Topology Appl.} \textbf{305} (2022), Article 107899.
		\newblock \url{https://doi.org/10.1016/j.topol.2021.107899}.
		
		\bibitem{Fremlin2001}
		D.~H. Fremlin,
		\newblock \emph{Measure Theory. Volume 2: Broad Foundations},
		\newblock Torres Fremlin, Colchester, 2001,
		\newblock \url{https://www1.essex.ac.uk/maths/people/fremlin/mt.htm}.
		
		\bibitem{GrevenPfaffelhuberWinter2009}
		A.~Greven, P.~Pfaffelhuber, and A.~Winter,
		\newblock Convergence in distribution of random metric measure spaces
		($\Lambda$-coalescent measure trees),
		\newblock \emph{Probab. Theory Related Fields} \textbf{145} (2009),
		285--322,
		\newblock \url{https://doi.org/10.1007/s00440-008-0169-3}.
		
		\bibitem{Gromov1999}
		M.~Gromov,
		\newblock \emph{Metric Structures for Riemannian and Non-Riemannian Spaces},
		\newblock Progress in Mathematics, vol.~152, Birkh\"auser, Boston, 1999,
		\newblock \url{https://www.ihes.fr/~gromov/distancegeometry/112/}.
		
		\bibitem{Gross2025}
		R.~Gross,
		\newblock The distance problem via subadditivity,
		\newblock \emph{Ann. Henri Lebesgue} \textbf{8} (2025), 1023--1035,
		\newblock \url{https://doi.org/10.5802/ahl.254}.
		
		\bibitem{Kechris1995}
		A.~S. Kechris,
		\newblock \emph{Classical Descriptive Set Theory},
		\newblock Graduate Texts in Mathematics, vol.~156,
		Springer, New York, 1995,
		\newblock \url{https://doi.org/10.1007/978-1-4612-4190-4}.
		
		\bibitem{Kondo2005}
		T.~Kondo,
		\newblock Probability distribution of metric measure spaces,
		\newblock \emph{Differential Geom. Appl.} \textbf{22} (2005), 121--130,
		\newblock \url{https://doi.org/10.1016/j.difgeo.2004.10.001}.
		
		\bibitem{MemoliNeedham2022}
		F.~M\'emoli and T.~Needham,
		\newblock Distance distributions and inverse problems for metric measure
		spaces,
		\newblock \emph{Stud. Appl. Math.} \textbf{149} (2022), 943--1001,
		\newblock \url{https://doi.org/10.1111/sapm.12526}.
		
		\bibitem{Vershik1998}
		A.~M. Vershik,
		\newblock The universal Urysohn space, Gromov metric triples and random
		metrics on the natural numbers,
		\newblock \emph{Russian Math. Surveys} \textbf{53} (1998), 921--928,
		\newblock \url{https://doi.org/10.1070/RM1998v053n05ABEH000069}.
		
		\bibitem{Vershik2004}
		A.~M. Vershik,
		\newblock Random and universal metric spaces,
		\newblock in \emph{Dynamics and Randomness II}, Nonlinear Phenomena and
		Complex Systems, vol.~10, Springer, Dordrecht, 2004, pp.~199--228,
		\newblock \url{https://doi.org/10.1007/978-1-4020-2469-6_6}.
		
	\end{thebibliography}
\end{document}